\documentclass[a4paper,12pt,reqno]{amsart}
\usepackage[utf8]{inputenc}
\usepackage{lmodern,amsmath,amssymb,mathtools,amsthm}
\usepackage{amsthm}
\usepackage{wasysym, enumerate, hyperref}
\usepackage{graphics,graphicx,geometry}
\usepackage{xcolor}
\usepackage{mathrsfs}
\usepackage{cleveref}
\usepackage{floatrow,morefloats}
\usepackage{epsfig}
\usepackage{epstopdf}
\usepackage{caption,subcaption}
\usepackage{exscale, relsize}
\usepackage{mathtools}
\usepackage{multirow}
\usepackage{ifthen}
\usepackage{siunitx}

\DeclarePairedDelimiter\floor{\lfloor}{\rfloor}
\usepackage{float}
\usepackage{placeins}
\usepackage{booktabs}
\usepackage{subcaption}

\numberwithin{equation}{section}
\newtheorem{definition}{Definition}[section]

\newtheorem{lemma}{Lemma}[section]
\newtheorem{theorem}{Theorem}[section]
\newtheorem{proposition}{Proposition}
\newtheorem{corollary}{Corollary}[section]
\newtheorem{remark}{Remark}[section]

\title{ Sobolev Orthogonal Polynomials on the unit circle with varying inner product
}

\author{Sourav Majumder$^\dagger$}
\address{$^\dagger$Department of Mathematics\\ Indian Institute of Technology, Roorkee-247667, Uttarakhand, India}
\email{souravmajumder93@gmail.com, sourav\_m@ma.iitr.ac.in}

\author{A. Swaminathan$^\ddagger$}{\thanks{$^{\#}$Corresponding author}}
\address{$^\ddagger$Department of Mathematics\\ Indian Institute of Technology, Roorkee-247667, Uttarakhand, India}
\email{mathswami@gmail.com, a.swaminathan@ma.iitr.ac.in}
\allowdisplaybreaks
\bigskip
\begin{document}
	
	
	\begin{abstract}
		In this work, we investigate the  validity of the Szeg\"o type strong asymptotic formula for the Sobolev orthogonal polynomials associated with a Sobolev-type inner product depending on a parameter $s$, which is defined as
		\begin{equation*}
			\langle f,g\rangle _{\frac{\lambda}{n^s}}:= \int_{\mathbb{T}} f(z)\overline{g(z)} d\mu_1(z)+ \dfrac{\lambda}{n^s} \int_{\mathbb{T}} f^{\prime}(z)\overline{g^{\prime}(z)} d\mu_2(z), \quad \text{where $s$ is fixed and } \quad s\geq 2,
		\end{equation*}
		and $\lambda >0$.
		Here, $\mu_1$, $\mu_2$ are absolutely continuous with respect to the normalized lebesgue measure  and $\int_{\mathbb{T}} \mu^{\prime}_i(z) dm=1$, where $\mu_i^{\prime}(z)= \dfrac{d\mu_i}{dm} >0$ for $i=1,2$.
		Further,  $\mu_1, \mu_2$ belongs to the Szeg\"o class and also has  non-zero bounded Radon-Nikodym derivatives on the unit circle.
		We observe that the measure appearing in the Szeg\"o function of the strong asymptotic limit depends on $s$. As a specific case, we consider the Bernstein-Christoffel Sobolev orthogonal polynomials and study their properties, including the asymptotic distribution and location of their zeros for different values of $s$. We further establish a spectral characterization of their zeros as the eigenvalue of an explicitly constructed modification of the associated Hessenberg matrix and provide numerical evidence supporting the obtained asymptotic and zeros results.
	\end{abstract}
	
	\maketitle
	\markboth{Sourav Majumder and A. Swaminathan}{Strong Asymptotics}

	

	\section{Introduction}
	\subsection{Preliminaries:}
	Let $\mathbb{D}=\{ z\in \mathbb{C}: \left|z\right|<1\}$ be the open unit disc, and let $\mathcal{M}$ be the finite positive Borel measures on the unit circle $\mathbb{T}=\{z\in \mathbb{C} : \left|z\right|=1\}$, with infinite points in its support. It is a very well-known fact \cite{Geronimus-Book, OPUCVol1-Simon Barry, Szego-book} that there exists a sequence of polynomials $\{\phi_n\}_{n\geq 0}$, with deg $\phi_n= n$,  such that
	\begin{align*}
		\int_{\mathbb{T}}\phi_m(z,d\mu)\overline{\phi_n(z,d\mu)} d\mu(z)=\delta_{m,n},
	\end{align*}
	where $\delta_{m,n}$ is the Kronecker delta.  $\{\phi_n(z,d\mu)\}_{n\geq 0}$ is said to be the sequence of orthonormal polynomials with respect to the measure $d\mu$, and we denote it  by
	\begin{align*}
		\phi_n(z;d\mu)= \kappa_n(d\mu)z^n+ \text{lower degree terms}, \quad \text{where}\quad\kappa_n(d\mu)>0.
	\end{align*}
	The corresponding sequence of monic orthogonal polynomials  $ \kappa_n(d\mu)^{-1} \phi_n(z;d\mu)$ will be denoted by $\left\{\Phi_n(z,d\mu)\right\}_{n \geqslant 0}$. They satisfy the following forward and backward recurrence relations \cite{OPUCVol1-Simon Barry,Szego-book},
$$
	\begin{aligned}
		& \Phi_{n+1}(z,d\mu)=z \Phi_n(z,d\mu)+\Phi_{n+1}(0,d\mu) \Phi_n^*(z,d\mu), \quad n \geqslant 0, \\
		& \Phi_{n+1}(z,d\mu)=\left(1-\left|\Phi_{n+1}(0,d\mu)\right|^2\right) z \Phi_n(z,d\mu)+\Phi_{n+1}(0,d\mu) \Phi_{n+1}^*(z,d\mu), \quad n \geqslant 0,
	\end{aligned}
$$
where $\Phi_n^*(z,d\mu)=z^n \overline{\Phi_n(1 / \bar{z},d\mu)}$, is the reversed polynomial, and $\left\{-\overline{\Phi_{n+1}(0,d\mu)}\right\}_{n \geqslant 0}$ are the so-called Verblunsky coefficients or reflection coefficients. These verblunsky coefficients play an important role in the theory of orthogonal polynomials on the unit circle (\textbf{OPUC}) as they completely determine the whole sequence $\{ \Phi_{n}(z, d\mu)\}$.
	\noindent From  the literature of $\bf{OPUC}$, we consider two important families of measures given by the following:
\begin{definition}
		 $\mu$ belongs to the Nevai class $\mathcal{N}$ \cite{JAT1982-MNT-Remarks on Rakhmanov's paper, Constructive approximation1987- Mate NevaiTotik- extension of Szego theory}, if
\begin{align*}
			\lim _{n \rightarrow \infty}\left|\Phi_{n+1}(0,d\mu)\right|=0.
		\end{align*}
\end{definition}
\begin{definition}
		$\mu$ belongs to the Szeg\"o class $\mathcal{S}$  \cite{Geronimus-Book, OPUCVol1-Simon Barry},  if
		\begin{align*}
			\log\mu^{'}\in L^1[0,2\pi].
		\end{align*}
\end{definition}
On the other hand, if $\mu \in \mathcal{S}$, then the Szeg\"o function, $D\left( \mu^{\prime},z,d\mu\right)$ is defined by
\begin{align*}
		D\left(\mu^{\prime},z;d\mu\right)=\exp \left(\dfrac{1}{4 \pi} \int_{0}^{2\pi} \dfrac{\mathrm{e}^{\mathrm{i} \theta}+z}{\mathrm{e}^{\mathrm{i} \theta}-z} \log \mu^{\prime}(\theta)  d\theta \right) , \quad z\in \mathbb{D}
\end{align*}
and
\begin{align}\label{sa 1.1}
		\lim_{n \to \infty} \kappa_n(d\mu)= \exp\left( -\dfrac{1}{4 \pi}\int_{0}^{2\pi} \log \mu^{'}(\theta) d\theta\right)= D\left(\mu^{\prime},0;d\mu\right)^{-1}.
\end{align}
We provide another notion of Nevai class here \cite{JAT2011-MartinezSimon-L2DerivativeOPUC}. If $\mu \in \mathcal{S}$, then it has normal $L^2$-derivative behaviour, i.e.
\begin{align*}
		\lim _{n \rightarrow \infty}\left(\int \dfrac{\left|\phi_n^{\prime}\left(\mathrm{e}^{\imath \theta}, d\mu \right)\right|^2}{n^2}  d\mu(\theta)\right)^{1 / 2}=1
\end{align*}
and hence $\mu \in \mathcal{N}$. On the other hand, if $\mu \in \mathcal{N}$, then
\begin{align*}
		\left|\dfrac{\Phi_{n}(z,d\mu)}{\Phi_{n-1}(z,d\mu)} -z\right| \leq \left|\Phi_{n}(0,d\mu)\right|, \quad \left|z\right|\geq 1.
\end{align*}
Thus, $\displaystyle\lim _{n \rightarrow \infty}\left(\dfrac{\Phi_{n}(z,d\mu)}{\Phi_{n-1}(z,d\mu)}\right) = z$, uniformly on compact subsets of $\left|z\right| >1$. In this case, the asymptotics was obtained under much weaker conditions. A well-known result \cite{Mat.sb1982- E.A. Rakhmanov-Ratio asymptotics} says that any measure that satisfies $\mu^{'}(z)>0 $ almost everywhere for $z \in \mathbb{T}$, satisfies $\displaystyle\lim_{n \to \infty}|\Phi_{n+1}(0,d\mu)| =0$. Now if $\mu \in \mathcal{S} $, there is a well known asymptotic formula for $\phi_n(z;d\mu)$ and $\phi_n^{*}(z,d\mu)$ given by
\begin{align*}
		\lim_{n \to \infty} \dfrac{\phi_n(z,d\mu)}{z^n} = \overline{D\left(\mu^{'}, \bar{z}^{-1}; d\mu\right)^{-1}} \quad \text{ uniformly for } |z|\geq R>1
\end{align*}
and
\begin{align*}
		\lim_{n \to \infty} \phi_n^{*}(z,d\mu) = D(\mu^{'}, z ; d\mu)^{-1}
\end{align*}
uniformly on compact subsets of $\mathbb{D}$ \cite{Szego-book}. Later, in \cite{Constructive approximation1987- Mate NevaiTotik- extension of Szego theory}, these results are extended  for a more general case. In fact, it is shown that, if $\mu \in \mathcal{N} $, $g$ be a nonnegative $d \mu$-integrable function, and  there exists a polynomial $\tilde{Q}$ such that $\tilde{Q} g$ and $\tilde{Q} g^{-1}$ are Riemann integrable, then for $ d\nu = g d \mu$, it is shown that  $\nu \in \mathcal{N}$. Moreover, given $0<r<1$, it is  proved that
\begin{align*}
		\lim _{n \rightarrow \infty} \dfrac{\phi_n^*( z,d \nu)}{\phi_n^*(z,d \mu)}=D\left(g^{-1}, z,d\nu\right), \quad \text{uniformly for }\quad |z| \leq r,
\end{align*}
and
\begin{align*}
\lim _{n \rightarrow \infty} \dfrac{\phi_n(z,d\nu)}{\phi_n(z,d\mu)}=\overline{D\left(g^{-1}, \bar{z}^{-1},d\nu\right)}
\end{align*}
holds uniformly for $|z| \geq s>1.$
	
\subsection{Literature survey and motivation:}
Over the last three decades, considerable attention has been devoted to extending Szeg\"{o}s theory of strong asymptotics for orthogonal polynomials to the framework of Sobolev orthogonal polynomials, which are orthogonal with respect to a Sobolev type inner product of the form
\begin{align}\label{sa 1.2}
		\langle f,g\rangle _{s}:= \int_{\Gamma} f(z)\overline{g(z)} d\mu_1(z)+  \int_{\Gamma} f^{\prime}(z)\overline{g^{\prime}(z)} d\mu_2(z),
\end{align}
where each $\mu_i$, $i=1,2$ is a finite positive Borel measure supported on a Jordan curve or arc.

The first significant contribution in this direction is given in  \cite{JCAM1998- Martinez-Asymptotics of the sobolev op} (see also \cite{constuctive approx 2000- Martinez- Bernstein Szego thm for sobolev op}), where a Szeg\"{o} type strong asymptotic formula is established for the corresponding Sobolev orthogonal polynomials with respect to the inner product \eqref{sa 1.2}. However, in these works \cite{JCAM1998- Martinez-Asymptotics of the sobolev op, constuctive approx 2000- Martinez- Bernstein Szego thm for sobolev op}, both $\mu_1$ and $\mu_2$ are measures taken from the Szeg\"{o} class. Here, the Szeg\"{o} function appearing  in the asymptotic expression depends solely on the second measure $\mu_2$. Later in \cite{JAT1999- Aptekarev Berriochoa Cachafeiro-Strong asymptotics}, the same asymptotic behaviour is proved on the unit circle without any assumption on $\mu_1$, while the second measure $\mu_2$, is required to satisfy additional smoothness condition stronger than the one given in  used in \cite{JCAM1998- Martinez-Asymptotics of the sobolev op, constuctive approx 2000- Martinez- Bernstein Szego thm for sobolev op} (see also \cite{JCAM2001- Martinez- analytic aspects of sobolev op}). In this situation also  the Szeg\"{o} function appearing  in the asymptotic expression depends on the second measure $\mu_2$. Beyond the results described above, several subsequent developments deserve attention. In 2011, inspired by the works \cite{JAT1999-Alfaro-Martinez et al.- asymptotic for balanced extremal for SOP} and \cite{TAMS2009- Alfaro et al- SOP with balance asymptotics}, the following theorem was established in \cite{JAT2011-MartinezSimon-L2DerivativeOPUC}.
\begin{theorem}{\cite{JAT2011-MartinezSimon-L2DerivativeOPUC}}\label{JAT: Properties of s=2 with same measure}
		Let $\{S_n(z)\}$ be the sequence of monic Sobolev orthogonal polynomials with respect to the inner product
\begin{align}\label{sa 1.3}
			\langle f,g\rangle _{\frac{\lambda}{n^2}}:= \int_{\mathbb{T}} \overline{f(z)}g(z) d\mu(z)+ \dfrac{\lambda}{n^2} \int_{\mathbb{T}} \overline{f^{\prime}(z)}g^{\prime}(z) d\mu(z)
\end{align}
and $\{ \Phi_n(z;d\mu)\}$ be the sequence of monic orthogonal polynomials with respect to the measure $\mu$. Suppose that
\begin{itemize}
			\item[(a)]  $\mu$ has normal derivative behaviour.
			\item[(b)] $\mu$ is in the Szeg\"o class.
\end{itemize}
Then
\begin{enumerate}[{\rm(1)}]
			\item  $\lim_{n \to \infty}$ $\dfrac{\sigma_n^2}{\| \Phi_{n}(z;d\mu)\|^2} =1+\lambda$, where
			$ \sigma_n = \min \left\{ \|P\|_{\frac{\lambda}{n^2}} \big|  P(z)= z^n+\ldots \right\}$ .
			\item $\lim_{n \to \infty} \| S_n(z)-\Phi_n(z,d\mu)\|^2_{\frac{\lambda}{n^2}}=0$.
			\item  On compact subsets of $\mathbb{C}\setminus \overline{\mathbb{D}}$,
\begin{align*}
\dfrac{S_n(z)}{\Phi_n(z, d\mu)} \to 1,
\end{align*}
uniformly.
\end{enumerate}
\end{theorem}

The motivation for defining the inner product \eqref{sa 1.3} is explicitly stated in \cite{JAT1999-Alfaro-Martinez et al.- asymptotic for balanced extremal for SOP}. As the Sobolev orthogonality is defined by combining the usual $L^2$-inner product with a derivative term, for polynomials of increasing degree, these two contributions have different natural scales, since differentiation multiplies the leading coefficient of a degree $n$ polynomial by $n$. Thus, in asymptotic problems, it is natural to consider a degree-dependent coefficient in the derivative part of the Sobolev inner product. In \cite{JAT1999-Alfaro-Martinez et al.- asymptotic for balanced extremal for SOP, TAMS2009- Alfaro et al- SOP with balance asymptotics, JAT2011-MartinezSimon-L2DerivativeOPUC}, the inner product defined by \eqref{sa 1.3}, is referred to as a balanced Sobolev inner product.
	Theorem \ref{JAT: Properties of s=2 with same measure} motivates us to the following questions.
\begin{itemize}
\item To study the asymptotic behaviour of the sequence of orthogonal polynomials associated with the inner product \eqref{sa 1.3}  by considering two distinct measures $\mu_1$ and $\mu_2$ under the same conditions given in Theorem \ref{JAT: Properties of s=2 with same measure}.
\item To analyze the asymptotic behaviour by increasing the decay in the Sobolev part and consider two distinct measures satisfying the same conditions as in Theorem \ref{JAT: Properties of s=2 with same measure}.
\end{itemize}
	To address these questions, we consider the sequence of Sobolev inner products
	\begin{align}\label{sa 1.4}
		\langle f,g\rangle _{\frac{\lambda}{n^s}}:= \int_{\mathbb{T}} f(z)\overline{g(z)} d\mu_1(z)+ \dfrac{\lambda}{n^s} \int_{\mathbb{T}} f^{\prime}(z)\overline{g^{\prime}(z)} d\mu_2(z),
	\end{align}
	where $\lambda >0$ and $s \geq 2$ be fixed. We impose the following conditions on the measures $\mu_1$ and $\mu_2$ in \eqref{sa 1.4}:
\begin{itemize}
		\item $\mu_1$, $\mu_2$ are absolutely continuous with respect to the normalized lebesgue measure. 
		$
		$
		\item $\int_{|z|=1} \mu^{\prime}_i(z) dm=1$, where $\mu_i^{\prime}(z)= \dfrac{d\mu_i}{dm} >0$ for $i=1,2$.
		\item  $\mu_1, \mu_2$ belongs to the Szeg\"o class.
		\item  The Radon-Nikodym derivative of the measures $\mu_1$ and $ \mu_2$ satisfies
\begin{align}\label{boundedness of the Radon-Nikodym derivatives}
			0<c_j \leq \mu_{j}^{\prime}(e^{\iota\theta})\leq C_j < \infty, \quad \text{for} \quad j=1,2,
\end{align}
		where $c_j$ and $C_j$, $j=1,2$ are appropriately chosen constants, which means
\begin{align*}
			0 <c \leq \dfrac{\mu^{\prime}_1(e^{\iota \theta})}{\mu^{\prime}_2(e^{\iota \theta})}\leq C < \infty.
\end{align*}
\end{itemize}
The above inequality implies that the  $L^2 $- norms of the measures $\mu_1$ and $\mu_2$, namely  $\|.\|_{L^2(d\mu_1)}$ and $\|.\|_{L^2(d\mu_2)}$ are equivalent, and both the measures $\mu_1$, $\mu_2$ have normal $L^2$-derivative behavior (see Theorem $5.1$ in \cite{JAT2011-MartinezSimon-L2DerivativeOPUC}).
Here, the corresponding monic orthogonal polynomial $Q_n(z)$ of degree $n$ can be characterized as the unique monic polynomial of degree $n$ that minimizes the norm induced by the inner product \eqref{sa 1.4} among all monic polynomials of degree $n$.
	Equivalently,
\begin{align*}
		\|Q_n(z)\|_{\frac{\lambda}{n^s}}= \min_{P\in \mathbb{P}_n} \| P\|^2_{\frac{\lambda}{n^s}}, \qquad \mathbb{P}_n:= \left\{ P : P(z)= z^n+ \ldots \right\}.
\end{align*}
	
Having established the necessary preliminaries, we organize the remainder of the paper as follows. In Section $2$, we prove a Szeg\"{o}-type strong asymptotic formula for the case $s>2$, and also study the corresponding results related to the strong asymptotics.  Section 3 is devoted to the special case $s=2$. In Section $4$, we consider an illustrative example, namely the Bernstein–Christoffel Sobolev orthogonal polynomials and derive a connection formula, and show how the strong asymptotic formula can be applied to determine the asymptotic distribution of their zeros in both cases $s>2$ and $s=2$. We exhibit a method to find the zeros of these Sobolev orthogonal polynomials as the eigenvalues of a explicitly constructed single line modification of the $n \times n$  Hessenberg matrix associated with the orthogonal polynomial sequence $\left\{\Phi_{n}(z; d\mu_1)\right\}$. Along with that we also study the theoretical and numerical analysis of the locations of their zeros for both $s>2$ and $s=2$ cases. We also provide support to the asymptotic results established in sections $2$ and $3$ for this specific orthogonal polynomial through numerical values.
	
\section{ Strong Asymptotics  of the Sobolev Orthogonal Polynomials for  \texorpdfstring{$s>2$}{s>2} case.}
One of the main objectives of this section is to establish the strong asymptotics for the sequence of Sobolev orthogonal polynomials $\left\{\Psi_n(z)\right\}_{n\geq 0}$ associated with the Sobolev inner product in \eqref{sa 1.4}, where $s>2$ is fixed and $\lambda>0$. To this end, we begin with the following results:
\begin{lemma}\label{Norm convergence of Poly difference of Phi and Psi, s>2}
		Let $\{ \Phi_n(z;d\mu_1)\}$ be the sequence of monic orthogonal polynomials with respect to the measure $\mu_1$. Then the following hold:
		\begin{align*}
			\| \Psi_n(z)-\Phi_n(z;d\mu_1)\|_{\frac{\lambda}{n^s}} \to 0 \quad \text{as} \quad n \to \infty  .
		\end{align*}
\end{lemma}
\begin{proof}
Write
\begin{align}\label{sa lemma 1 1}
			\Psi_n(z)=\Phi_{n}(z;d\mu_1)-Q_{n-1}(z),
\end{align}
		where $Q_{n-1}(z)$ is a polynomial of degree $n-1$. Now
\newline
$
\displaystyle
\langle \Psi_n(z), \Psi_n(z)\rangle _{d\mu_1}
$
\begin{align*}
&=\langle \Phi_{n}(z;d\mu_1)-Q_{n-1}(z),\Phi_{n}(z;d\mu_1)-Q_{n-1}(z)\rangle_{d\mu_1}\\
&=\|\Phi_n(z;d\mu_1)\|^2_{L^2(d\mu_1)}+\|Q_{n-1}(z)\|^2_{L^2(d\mu_1)}-2 \text{Re}\langle \Phi_n(z;d\mu_1),Q_{n-1}(z)\rangle_{L^2(d\mu_1)}
\end{align*}
		and $\Phi_n(z;d\mu_1) \perp Q_{n-1}(z)$ in $L^2(d\mu_1)$ implies
\begin{equation*}
			\| \Psi_n(z)\|^2 _{L^2(d\mu_1)}= \|\Phi_n(z;d\mu_1)\|^2_{L^2(d\mu_1)}+\|Q_{n-1}(z)\|^2_{L^2(d\mu_1)}.
\end{equation*}
Since
\begin{align*}
\dfrac{\lambda}{n^s}\int_{|z|=1} | \Psi_n^{\prime}(z)|^2 d\mu_{2}(z)
&
=\dfrac{\lambda}{n^s}\int_{|z|=1} |\Phi_{n}^{\prime}(z,d\mu_1)|^2d\mu_{2}+\dfrac{\lambda}{n^s}\int_{|z|=1} |Q^{\prime}_{n-1}(z)|^2 d\mu_{2}
\\
&\qquad
-\dfrac{2\lambda}{n^s}\text{Re} \left(\int_{|z|=1} \Phi_n^{\prime}(z;d\mu_1)\overline{Q^{\prime}_{n-1}(z)} d\mu_2\right),
\end{align*}
we can write
\newline
$
\displaystyle
\|\Psi_n(z)\|^2_{\frac{\lambda}{n^s}}
$
\begin{align}\label{sa lemma1 2}
			= \|\Phi_n(z;d\mu_1)\|^2_{\frac{\lambda}{n^s}}+\|Q_{n-1}(z)\|^2_{\frac{\lambda}{n^s}}-\dfrac{2\lambda}{n^s}\text{Re}\left(\int_{|z|=1}\Phi_n^{\prime}(z,d\mu_1)\overline{Q_{n-1}^{\prime}(z)}d\mu_{2}\right)	\end{align}
		Now by using  the minimal norm property and \eqref{sa lemma1 2}, we have
\begin{align*}
\|\Phi_n(z;d\mu_1)\|^2_{\frac{\lambda}{n^s}}+\|Q_{n-1}(z)\|^2_{\frac{\lambda}{n^s}}-\dfrac{2\lambda}{n^s}\text{Re}\left(\int_{|z|=1}\Phi_n^{\prime}(z,d\mu_1)\overline{Q_{n-1}^{\prime}(z)}d\mu_{2}\right) \leq \|\Phi_n(z;d\mu_1)\|^2_{\frac{\lambda}{n^s}},
	\end{align*}
	which implies
\begin{equation}\label{sa lemma1 3}
		\|Q_{n-1}(z)\|^2_{\frac{\lambda}{n^s}} -\dfrac{2\lambda}{n^s}\text{Re}\left(\int_{|z|=1}\Phi_n^{\prime}(z,d\mu_1)\overline{Q_{n-1}^{\prime}(z)}d\mu_{2}\right)	\leq 0.
\end{equation}
By Cauchy-Schwarz inequality, we have
\begin{align*}
		\text{Re}\left(\int_{|z|=1}\Phi_n^{\prime}(z,d\mu_1)\overline{Q_{n-1}^{\prime}(z)}d\mu_{2}\right) &\leq \|\Phi_n^{\prime}(z,d\mu_1)\|_{L^2(d\mu_2)} \| Q^{\prime}_{n-1}(z)\|_{L^2(d\mu_{2})}.
\end{align*}
	Normal derivative behavior of $\mu_1$ implies
\begin{align}\label{sa  lemma1 4}
		\| \phi^{\prime}_n(z;d\mu_1)\|_{L^2(d\mu_1)}=n(1+o(1)).
\end{align}
Since $ \Phi_n(z;d\mu_1)=\kappa_n(d\mu)^{-1} \phi_n(z;d\mu)$,  using \eqref{sa 1.1}, \eqref{sa  lemma1 4} and
\begin{align*}
		\|\Phi^{\prime}_n(z;d\mu_1)\|_{L^2(d\mu_1)}= \newline \kappa_n(d\mu)^{-1} \|\phi^{\prime}_n(z;d\mu)\|_{L^2(d\mu_1)}
\end{align*}	
we conclude that
\begin{align*}
		\|\Phi^{\prime}_n(z;d\mu_1)\|_{L^2(d\mu_1)}=n(1+o(1)).
\end{align*}
Using the above argument along with the norm equivalance of the measures $\|.\|_{L^2(d\mu_1)}$ and $\|.\|_{L^2(d\mu_2)}$, we get
\begin{equation*}
		\|\Phi^{\prime}_n(z;d\mu_1)\|_{L^2(d\mu_2)}=\mathcal{O}(n).
\end{equation*}
	This means  there exist a constant $M^{\prime}$ such that $\|\Phi^{\prime}_n(z;d\mu_1)\|_{L^2(d\mu_2)} \leq M^{\prime}n$ for all $n \ \geq N_0$. Now by Bernstein  inequality we  have
\begin{align}\label{sa lemma1 5}
		\|Q_{n-1}^{\prime}(z)\|_{L^2(d\mu_2)}\leq K_2 n \|Q_{n-1}(z)\|_{L^2(d\mu_2)},
\end{align}
where $K_2 >0$.
Using the norm equivalence between   $\| .\|_{L^2(d\mu_1)}$ and $\|. \|_{L^2(d\mu_2)}$ we get
\begin{align} \label{sa lemma1 6}
		\|Q_{n-1}(z)\|_{L^{2}(d\mu_2)} \leq C \|Q_{n-1}(z)\|_{L^2(d\mu_1)}\leq C \|Q_{n-1}(z)\|_{\frac{\lambda}{n^s}},
\end{align}
	for some $C>0$.
	Using \eqref{sa lemma1 6}, \eqref{sa lemma1 5} becomes
\begin{align*}
		\|Q_{n-1}^{\prime}(z)\|_{L^2(d\mu_2)}\leq K_3 n \|Q_{n-1}(z)\|_{\frac{\lambda}{n^s}}, \quad \text{where} \quad K_3= CK_2.
\end{align*}
	Hence, for $n \geq N_0$, we have
\begin{align}\label{sa lemma1 7}
		\dfrac{2\lambda}{n^s}\text{Re}\left(\int_{|z|=1}\Phi_n^{\prime}(z,d\mu_1)\overline{Q_{n-1}^{\prime}(z)}d\mu_{2}\right)\leq \dfrac{2\lambda}{n^s}K_4 n^2 \|Q_{n-1}(z)\|_{\frac{\lambda}{n^s}},
\end{align}
	where $K_4 =  M^{\prime} K_3$.
	From \eqref{sa lemma1 3} and \eqref{sa lemma1 7} we have
\begin{align*}
		\|Q_{n-1}(z)\|^2_{\frac{\lambda}{n^s}}\leq
		\dfrac{2\lambda}{n^s}K_4 n^2 \|Q_{n-1}(z)\|_{\frac{\lambda}{n^s}}.
\end{align*}
	$ \|Q_{n-1}(z)\|_{\frac{\lambda}{n^s}}\neq 0$ implies
\begin{align}\label{sa lemma1 8}
		0<\|Q_{n-1}(z)\|_{\frac{\lambda}{n^s}}\leq \dfrac{2\lambda}{n^s}K_4 n^2.
\end{align}
Since $s>2$, we can write $s=2+\epsilon$ for any $\epsilon>0.$  Using this in \eqref{sa lemma1 8}, and letting $n \to \infty $, gives the  final result.
\end{proof}
	
\begin{lemma}\label{Bound for classical Phi w.r.t first measure}
			On the compact subset $K$ of $|z|>1$, we have
\begin{enumerate}
				\item [\rm (1)]$\exists$ a constant $C_k>0$	 such that $	| \Phi_n(z,d\mu_1)| \geq C_k |z|^n $.
				\item [\rm (2)] $\exists$ a constant  $L_k>0$ such that 	$|\Phi_{n}(z;d\mu_1)| \leq L_kR^n $  where $ R:= \max_{z\in K} |z|,$
\end{enumerate}
			where $\left\{ \Phi_n(z; d\mu_1)\right\}$ is the sequence of monic orthogonal polynomials with respect to the measure $\mu_1$.
\end{lemma}
\begin{proof}
			Since $\mu_1$ is in the Szeg\"o class, we have
\begin{align}\label{sa lemma2 1}
				\lim_{n \to \infty} \dfrac{\Phi_n(z,d\mu_1)}{z^n} = \dfrac{\overline{D(	\mu^{\prime}_1,\bar{z}^{-1};d\mu_1)^{-1}}}{\overline{D(\mu^{\prime}_1,0;d\mu_1)^{-1}}}:=G(z)
\end{align}
uniformly on compact subsets of $|z|>1$. Now let $K\subset \{z : |z|>1\}$ be compact. Then $G(z)$ is continuous and non-zero  on $K$. This implies  $\exists$ $m_k >0$ such that $ |G(z)|\geq m_k$ for all $ z \in K$. Now if we choose   $\epsilon = \dfrac{m_k}{2}$, then using \eqref{sa lemma2 1}, we get
\begin{align}\label{sa lemma2 2}
				\left| \dfrac{\Phi_n(z,d\mu_1)}{z^n} -G(z)\right| \leq \dfrac{m_k}{2}, \quad \forall z \in K, \quad \forall n \geq N.
\end{align}
From \eqref{sa lemma2 2}, we get
\begin{align*}
				\left| \dfrac{\Phi_n(z,d\mu_1)}{z^n}\right| \geq \dfrac{m_k}{2}
\end{align*}
which implies
\begin{align} \label{sa lemma2 3}
				| \Phi_n(z,d\mu_1)| \geq \dfrac{m_k}{2}|z|^n.
\end{align}
Since $K $ is a compact subset contained in $|z|>1$, there exists $r>1$, such that $|z| \geq r$ for all $z \in K$. From \eqref{sa lemma2 3}, we get
\begin{align*}
				|\Phi_n(z;d\mu_1)| \geq C_k r^n  \quad \forall z \in K , \quad \text{where} \quad C_k:= \dfrac{m_k}{2}.
\end{align*}
Similarly, if we define $M_k:= \sup_{z\in K} \left|G(z)\right|$ and $\epsilon =\dfrac{m_k}{2}$, using \eqref{sa lemma2 2}, we get
\begin{align}\label{sa lemma2 4}
				\left| \dfrac{\Phi_n(z,d\mu_1)}{z^n}\right| \leq \left(\dfrac{m_k}{2}+M_k\right).
\end{align}
Further using \eqref{sa lemma2 4}, we have
\begin{align*}
|\Phi_{n}(z;d\mu_1)| \leq L_k R^n,
\end{align*}
where $R:= \max_{z\in K} |z| $, and $L_k= \dfrac{m_k}{2}+M_k$.
\end{proof}
\begin{lemma}\label{Uniform convergence of ratio of Psi and Phi, s>2}
			Let $\left\{\Psi_n(z)\right\}$ be the sequence of Sobolev orthogonal polynomial with respect to the inner product \eqref{sa 1.4} with $s$- fixed, $s>2$ and $\left\{\Phi_n(z;d\mu_1)\right\}$ be the sequence of monic orthogonal polynomials with respect to the measure $\mu_1$. Then
			\begin{align*}
				\dfrac{\Psi_n(z)}{\Phi_n(z, d\mu_1)} \to 1
\end{align*}
uniformly on the compact subsets of $\mathbb{C}\setminus \overline{\mathbb{D}}$.
\end{lemma}
\begin{proof}
			From \eqref{sa lemma 1 1}, we can write
\begin{align*}
				\dfrac{\Psi_n(z)}{\Phi_n(z;d\mu_1)}=1-\dfrac{Q_{n-1}(z)}{\Phi_n(z;d\mu_1)}.
\end{align*}		
			From Lemma \ref{Norm convergence of Poly difference of Phi and Psi, s>2}, we have $\|Q_{n-1}(z)\|_{L^2(d\mu_1)} \to 0$. So, we have,  $\|Q_{n-1}(z)\|_{L^2(\frac{dz}{2\pi\iota z})} \to 0$, by \eqref{boundedness of the Radon-Nikodym derivatives} . Now define $Q_{n-1}(z) := \sum_{k=0}^{n-1} a_{n,k} z^k$. By using Parseval's identity, we get
\begin{align*}
				\|Q_{n-1}(z)\|^2_{L^2(\frac{dz}{2\pi\iota z})}=\sum_{k=0}^{n-1}|a_{k,n}|^2 \to 0 \quad \text{as} \quad n \to \infty .
\end{align*}
			Now if $K \subset \{z : |z|>1\}$ is compact, we define  $R:= \min_{z\in K} |z|$, such that  $R>1$, and for every $ z \in K$, $|z| \geq R$. $ \dfrac{Q_{n-1}(z)}{z^{n-1}}= \sum_{k=0}^{n-1} a_{k,n} z^{k-(n-1)}$. Clearly, $|z|^{-(n-1-k)}\leq R^{-(n-1-k)}$. Therefore,
\begin{align*}
				\left| \dfrac{Q_{n-1}(z)}{z^{n-1}}\right| \leq \sum_{k=0}^{n-1} |a_{k,n}| R^{-(n-1-k)}.
\end{align*}
Applying Cauchy-Schwarz inequality, we get
\begin{align*}
				\sup_{z\in K} \left| \dfrac{Q_{n-1}(z)}{z^{n-1}}\right| \leq \left( \sum_{k=0}^{n-1}|a_{k,n}|^2\right)^{1/2} \left(\sum_{j=0}^{\infty} R^{-2j}\right)^{1/2}.
\end{align*}
Since $ \sum_{j=0}^{\infty} R^{-2j}< \infty$, we get
\begin{align}\label{sa lemma3 1}
				\sup_{z\in K}  	\left| \dfrac{Q_{n-1}(z)}{z^{n-1}}\right| \to 0 \quad \text{as} \quad n \to \infty.
\end{align}
Therefore, if we write
\begin{align*}
\dfrac{Q_{n-1}(z)}{\Phi_n(z;d\mu_1)}=  \dfrac{1}{z}\dfrac{Q_{n-1}(z)}{z^{n-1}} \dfrac{1}{\frac{\Phi_n(z,d\mu_1)}{z^n} },
\end{align*}
then  from  \eqref{sa lemma3 1} and Lemma  \ref{Bound for classical Phi w.r.t first measure}, we get
\begin{align}\label{sa lemma3 2}
				\dfrac{Q_{n-1}(z)}{\Phi_n(z;d\mu_1)} \to 0 \quad \text{uniformly on the compact subsets of } \mathbb{C}\setminus \overline{ \mathbb{D}}.
\end{align}
From \eqref{sa lemma3 2}, we get the result and the proof is complete.
\end{proof}
Now we are in a position to provide one of the main results of this section, which is the strong asymptotic formula for $\Psi_n(z)$.
		\begin{theorem}\label{Strong Asymptote for Psi, s>2 case}
			Let $\left\{\Psi_n(z)\right\}$ be the sequence of Sobolev orthogonal polynomial with respect to the inner product \eqref{sa 1.4} with $s$- fixed, $s>2$ and $\lambda >0$. Then
\begin{align*}
				\lim_{n \rightarrow \infty} \dfrac{\Psi_n(z)}{z^n}= \dfrac{\overline{D(	\mu^{\prime}_1,\bar{z}^{-1};d\mu_1)^{-1}}}{\overline{D(\mu^{\prime}_1,0;d\mu_1)^{-1}}}\quad \text{uniformly for compact subsets } K \quad \text{of} \quad |z|>1,
\end{align*}
			where $D(\mu^{'}_1,z;d\mu_1)$ is the Szeg\"o function of the measure $d\mu_1$, defined by
\begin{align*}
				D(\mu^{\prime}_1,z;d\mu_1)= \exp\left( \dfrac{1}{4\pi}\int_{0}^{2\pi} \dfrac{e^{\iota \theta}+z}{e^{\iota \theta}-z}\log(\mu_1^{\prime}(e^{\iota \theta}))d\theta\right).
\end{align*}
\end{theorem}
\begin{proof}
It is enough to show that
\begin{align*}
\left|\dfrac{\Psi_n(z)}{z^n}-\dfrac{\Phi_{n}(z;d\mu_1)}{z^n}\right|<\epsilon , \quad \forall n \geq N \quad \text{and}\quad z\in K\subset \mathbb{C}\setminus \overline{\mathbb{D}}.	
\end{align*}
From Lemma \ref{Uniform convergence of ratio of Psi and Phi, s>2}, we have
\begin{align}\label{sa thm1 2}
\left|\dfrac{\Psi_n(z)}{\Phi_{n}(z;d\mu_1)}-1\right| <\dfrac{\epsilon}{2L},  \quad \forall n \geq N_1 \quad \text{and} \quad z\in K.
\end{align}
Hence  $ \forall n \geq N_1 $ and $ z\in K,$ and using \eqref{sa lemma2 4} and \eqref{sa thm1 2} ,  we have
\begin{align}\label{sa thm1 3}
				\left|\dfrac{\Psi_n(z)}{z^n}-\dfrac{\Phi_n(z;d\mu_1)}{z^n}\right|&=\left|\dfrac{\Phi_n(z,d\mu_1)}{z^n}\right|\left|\dfrac{\Psi_n(z)}{\Phi_n(z;d\mu_1)}-1\right| \leq \dfrac{\epsilon}{2}.
\end{align}
The conclusion follows from using  \eqref{sa lemma2 1} and \eqref{sa thm1 3}.
\end{proof}

\begin{remark}\label{observation on SA, for s=2}
In  \cite{JAT1999- Aptekarev Berriochoa Cachafeiro-Strong asymptotics, constructive approx2003- berriochoa cachafeiro- Strong asymptotics for sobolev op, JCAM1998- Martinez-Asymptotics of the sobolev op}, the Szeg\"{o} function in the limit of strong asymptotic formula depends on the measure $\mu_2$ of the Sobolev inner product given by \eqref{sa 1.4}. By contrast, in this work, the strong asymptotic formula depends on the measure $\mu_1$ with respect to the inner product given in \eqref{sa 1.4}.
\end{remark}
\begin{remark}
For  large $n$, $\Psi_n(z)$ behaves like $\Phi_n(z,d\mu_1)$ on the compact subsets of $|z|>1$.
\end{remark}
Lemma \ref{Uniform convergence of ratio of Psi and Phi, s>2} and its consequences lead to the question of finding the asymptotic value of the ratio of $\Psi_n(z)$ and $\Phi_{n}(z; d\mu_2)$, which is a monic orthogonal polynomial of degree $n$ with respect to the measure $\mu_2$. This leads to the following result.
\begin{corollary}\label{Ratio of Psi and Phi in terms of szego function}
\begin{align*}
\lim_{n \to \infty} \dfrac{\Psi_n(z)}{\Phi_{n}(z;d\mu_2)} = \dfrac{\overline{D(	\mu^{\prime}_1,\bar{z}^{-1};d\mu_1)^{-1}}}{\overline{D(\mu^{\prime}_2,\bar{z}^{-1};d\mu_2)^{-1}}} \cdot \dfrac{\overline{D(\mu^{\prime}_2,0;d\mu_2)^{-1}}}{\overline{D(\mu^{\prime}_1,0;d\mu_1)^{-1}}}
\end{align*}
uniformly on the compact subsets of $|z|>1$.
\end{corollary}
\begin{proof}
Since
\begin{align*}
\dfrac{\Psi_n(z)}{\Phi_{n}(z;d\mu_2)}&= \dfrac{\Psi_n(z)}{\Phi_{n}(z; d\mu_1)} \cdot \dfrac{\dfrac{\Phi_{n}(z;d\mu_1)}{z^n}}{\dfrac{\Phi_{n}(z;d\mu_2)}{z^n}},
\end{align*}
the final result follows from  Lemma \ref{Uniform convergence of ratio of Psi and Phi, s>2}, and along with the fact that $\mu_1,\mu_2$ belongs to the Szeg\"o class.
\end{proof}
The above discussions lead to the following ratio asymptotic result.
\begin{theorem}\label{Ratio asymptotic for Psi, s>2}
$\displaystyle\lim_{n \to \infty} \dfrac{\Psi_{n+1}(z)}{\Psi_n(z)} =z $
			uniformly on the compact subsets of $|z|>1$, where $\left\{\Psi_n(z)\right\}$ be the sequence of Sobolev orthogonal polynomial with respect to the inner product \eqref{sa 1.4} with $s$- fixed, $s>2$ and $\lambda >0$.
\end{theorem}
\begin{proof}
The result  follows from Lemma \ref{Uniform convergence of ratio of Psi and Phi, s>2}, along with the  fact that $\mu_1 \in \mathcal{N} $ and
\begin{align*}
				\dfrac{\Psi_{n+1}(z)}{\Psi_n(z)}= \dfrac{\Psi_{n+1}(z)}{\Phi_{n+1}(z;d\mu_1)}\cdot \dfrac{\Phi_{n+1}(z;d\mu_1)}{\Phi_{n}(z;d\mu_1)} \cdot \dfrac{\Phi_{n}(z;d\mu_1)}{\Psi_n(z)}.
\tag*{{\qedhere}}
\end{align*}
\end{proof}
Using the strong asymptotic formula and a direct consequence of the  Hurwitz's theorem, we obtain the following corollary:
\begin{corollary}\label{Hurwitz, zeros of Psi, s>2}
For every $\epsilon >0$ and $n$ sufficiently large, the zeros of $\Psi_{n}(z)$ are in $|z|<1+\epsilon$.
\end{corollary}

The following result is a direct consequence of the strong asymptotic formulas established in this section and is stated without proof.
			
\begin{theorem}
Let  $\{\Psi_{n}(z)\}_{n\geq 0}$ be the sequence of monic Sobolev orthogonal polynomials for $s>2$, $s$-fixed with respect to the inner product of the form \eqref{sa 1.4}.
Let $K$ be the compact subset of $|z|>1$. Then the following holds:
\begin{enumerate}
					\item[\rm(1)] $\lim_{n \to \infty} \left( \sup_{z\in K} |\Psi_n(z)|\right)^{1/n} = \max_{z\in K} |z|$.
					\item[\rm(2)] $\lim_{n \to \infty} \left( \inf_{z\in K} |\Psi_n(z)|\right)^{1/n} = \min_{z\in K} |z|$.
\end{enumerate}
\end{theorem}

	
\section{ Strong Asymptotics  of the Sobolev Orthogonal Polynomials for  \texorpdfstring{$s=2$}{s=2} case.}
In this section, we will discuss the strong asymptotic result and its consequences for the Sobolev orthogonal polynomials with respect to the Sobolev  inner product
\begin{align}\label{sa ip s=2 3.1}
			\langle f,g\rangle _{\frac{\lambda}{n^2}}:= \int_{\mathbb{T}} f(z)\overline{g(z)} d\mu_1+\dfrac{\lambda}{n^2} \int_{\mathbb{T}} f^{\prime}(z)\overline{g^{\prime}(z)}d\mu_2
\end{align}
where $\mu_1$, $\mu_2$ are as defined in \eqref{sa 1.4}.  Before proceeding, we establish a preliminary result that is analogous to $ \mathrm{(1)}$ of Theorem \ref{JAT: Properties of s=2 with same measure}  (see also \cite{JAT2011-MartinezSimon-L2DerivativeOPUC}).
\begin{proposition}\label{Inequality of limsup/inf Ratio of min norm to Phi norm, s=2}
Let  $\{\tilde{\Psi}_n(z)\}_{n\geq0}$ be the corresponding sequence of monic orthogonal polynomials with respect to  \eqref{sa ip s=2 3.1}, where $\lambda >0$. If we define
\begin{align*}
				\sigma_n = \min \{ \|P\|_{\frac{\lambda}{n^2}} \| P(z)=z^n+ \ldots\}
\end{align*}
Then the following inequalities  hold:
\begin{itemize}
\item[\rm (1)] For constants $C_1>0$ and $C_2>0$, we have
\begin{align}\label{sa prop 1 1}
					1+\lambda C^2_1 \leq \liminf_{n \to \infty} \dfrac{\sigma^2_n}{\|\Phi_n(z;d\mu_1)\|_{L^2(d\mu_1)}^2}\leq \limsup_{n \to \infty} \dfrac{\sigma^2_n}{\|\Phi_n(z;d\mu_1)\|_{L^2(d\mu_1)}^2} \leq 1+ \lambda C^2_2.
\end{align}
\item[\rm (2)]  For constants $K_1>0$ and $K_2>0$, we have
\begin{align}\label{sa prop1 2}
					K_{1}^2+\lambda \leq \liminf_{n \to \infty} \dfrac{\sigma^2_n}{\|\Phi_n(z;d\mu_2)\|_{L^2(d\mu_2)}^2}\leq \limsup_{n \to \infty} \dfrac{\sigma^2_n}{\|\Phi_n(z;d\mu_2)\|_{L^2(d\mu_2)}^2} \leq K_{2}^2+ \lambda.
\end{align}
\end{itemize}
\end{proposition}
\begin{proof}
We start with the proof of the inequality \eqref{sa prop 1 1} of the theorem.
By the minimal norm property of the monic orthogonal polynomials, we have
\begin{align}\label{sa prop1 3}
					\| \tilde{\Psi}_n(z)\|^2_{L^2(d\mu_1)} \geq  \| \Phi_{n}(z;d\mu_1)\|^2_{L^2(d\mu_1)}
\end{align}
and
\begin{align}\label{sa prop1 4}
					\| \dfrac{\tilde{\Psi}^{\prime}_n(z)}{n}\|^2_{L^2(d\mu_2)} \geq  \| \Phi_{n-1}(z;d\mu_2)\|^2_{L^2(d\mu_2)}.
\end{align}
Since
\begin{align}\label{sa prop1 5}
					\sigma_n^2 = \| \tilde{\Psi}_n(z)\|^2_{\frac{\lambda}{n^2}} = \| \tilde{\Psi}_n(z)\|^2_{L^2(d\mu_1)}+\dfrac{\lambda}{n^2}\|\tilde{\Psi}^{\prime}_n(z)\|^2_{L^2(d\mu_2)},
\end{align}
using \eqref{sa prop1 3} and \eqref{sa prop1 4} in  \eqref{sa prop1 5}, we get
\begin{align}\label{sa prop1 6}
					\| \tilde{\Psi}_n(z)\|^2_{\frac{\lambda}{n^2}} \geq \| \Phi_{n}(z;d\mu_1)\|^2_{L^2(d\mu_1)}+ \lambda \| \Phi_{n-1}(z;d\mu_2)\|^2_{L^2(d\mu_2)}.
\end{align}
Since $\|.\|_{L^2(d\mu_1)}$ and $\|.\|_{L^2(d\mu_2)}$ are equivalent, there exist constants $C_1,C_2>0$ such that
\begin{align}\label{sa prop1 7}
					C_1 \|. \|_{L^2(d\mu_1)} \leq \|.\|_{L^2(d\mu_2)} \leq C_2 \|.\|_{L^2(d\mu_1)} .
\end{align}
				From  \eqref{sa prop1 6} and \eqref{sa prop1 7}, we get
\begin{align*}
					\| \tilde{\Psi}_n(z)\|^2_{\frac{\lambda}{n^2}} 	\geq \| \Phi_{n}(z;d\mu_1)\|^2_{L^2(d\mu_1)}+ \lambda C^2_1\| \Phi_{n-1}(z;d\mu_2)\|^2_{L^2(d\mu_1)},
\end{align*}
which using the  minimal norm property  can be written as
\begin{align}\label{sa prop1 8}
					\dfrac{	\| \tilde{\Psi}_n(z)\|^2_{\frac{\lambda}{n^2}} }{\| \Phi_{n}(z;d\mu_1)\|^2_{L^2(d\mu_1)}} \geq 1+ \lambda C_1^2 \dfrac{\| \Phi_{n-1}(z;d\mu_1)\|^2_{L^2(d\mu_1)}}{\| \Phi_{n}(z;d\mu_1)\|^2_{L^2(d\mu_1)}}.
\end{align}
Since $\mu_1 \in \mathcal{S}$ implies $\lim_{n \to \infty}  \dfrac{\| \Phi_{n-1}(z;d\mu_1)\|^2_{L^2(d\mu_1)}}{\| \Phi_{n}(z;d\mu_1)\|^2_{L^2(d\mu_1)}}=1$, from \eqref{sa prop1 8}, we get the left side of the inequality \eqref{sa prop 1 1}.
			
On the other hand, by using minimal norm property, we obtain
\begin{align*}
					\sigma_n^2 = 	\| \tilde{\Psi}_n(z)\|^2_{\frac{\lambda}{n^2}}  \leq \| \Phi_{n}(z;d\mu_1)\|^2_{L^2(d\mu_1)} +\dfrac{\lambda}{n^2} \| \Phi^{\prime}_{n}(z;d\mu_1)\|^2_{L^2(d\mu_2)}
\end{align*}
leading to
\begin{align*}
\dfrac{	\| \tilde{\Psi}_n(z)\|^2_{\frac{\lambda}{n^2}} }{\| \Phi_{n}(z;d\mu_1)\|^2_{L^2(d\mu_1)}} \leq 1+\lambda C_2^2 \left(\dfrac{\| \Phi^{\prime}_{n}(z;d\mu_1)\|_{L^2(d\mu_1)}}{n \| \Phi_{n}(z;d\mu_1)\|_{L^2(d\mu_1)}}\right)^2,
\end{align*}
which by the normal derivative behavior of $\mu_1 $ implies the right hand side of the inequality of \eqref{sa prop 1 1}. This completes the proof of \eqref{sa prop 1 1}.

Now we prove the second inequality \eqref{sa prop1 2} of the theorem.
Since $\|.\|_{L^2(d\mu_1)}$ and $\|.\|_{L^2(d\mu_2)}$ are equivalent, there exist constants $K_1,K_2>0$ such that
\begin{align*}
					K_1 \|. \|_{L^2(d\mu_2)} \leq \|.\|_{L^2(d\mu_1)} \leq K_2 \|.\|_{L^2(d\mu_2)} .
\end{align*}
Again, using the minimum norm property, we get $\| \tilde{\Psi}_n(z)\|^2_{\frac{\lambda}{n^2}}\leq \|\Phi_n(z;d\mu_2) \|^2_{\frac{\lambda}{n^2}} $. Expanding the right hand side of the inequality, we get
\begin{align}\label{RHS Sobolev norm expansion}
					\sigma_n^2 =	\|\tilde{\Psi}_n(z)\|^2_{\frac{\lambda}{n^2}}\leq \|\Phi_n(z;d\mu_2) \|^2_{L^2(d\mu_1)}+\dfrac{\lambda}{n^2} \|\Phi^{\prime}_n(z;d\mu_2) \|^2_{L^2(d\mu_2)}.
\end{align}
From \eqref{RHS Sobolev norm expansion}, we can write
\begin{align}\label{RHS inequality}
					\dfrac{\|\tilde{\Psi}_n(z)\|^2_{\frac{\lambda}{n^2}}}{\|\Phi_n(z;d\mu_2)\|^2_{L^2(d\mu_2)}}\leq \dfrac{\|\Phi_n(z;d\mu_2)\|^2_{L^2(d\mu_1)}}{\|\Phi_n(z;d\mu_2)\|^2_{L^2(d\mu_2)}}+ \dfrac{\lambda}{n^2} \dfrac{\|\Phi^{\prime}_n(z;d\mu_2)\|^2_{L^2(d\mu_2)}}{\|\Phi_n(z;d\mu_2)\|^2_{L^2(d\mu_2)}}.
\end{align}
Using norm equivalence and normal derivative behavior of $\mu_2$ in  \eqref{RHS inequality}, we have the right side inequality of \eqref{sa prop1 2}. Again, using minimum norm property and norm equivalence, we can write
\begin{align*}
					\|\tilde{\Psi}_n(z)\|^2_{L^2(d\mu_1)} & \geq \| \Phi_n(z;d\mu_1)\|^2_{L^2(d\mu_1)}\\
                                                          &\geq K^2_1 \| \Phi_n(z;d\mu_1)\|^2_{L^2(d\mu_2)} \geq K_1^2 \|\Phi_n(z;d\mu_2)\|^2_{L^2(d\mu_2)}
\end{align*}
leading to
\begin{align*}
					\|\tilde{\Psi}_n(z)\|^2_{\frac{\lambda}{n^2}}\geq  K_1^2 \|\Phi_n(z;d\mu_2)\|^2_{L^2(d\mu_2)} + \lambda \| \Phi_{n-1}(z;d\mu_2)\|^2_{L^2(d\mu_2)}.
\end{align*}
Now dividing both sides of the above by $\|\Phi_n(z;d\mu_2)\|^2_{L^2(d\mu_2)} $ and using the fact that $\mu_2 \in \mathcal{S}$, we get
\begin{align*}
					\liminf_{n \to \infty} \dfrac{\sigma^2_n}{\|\Phi_n(z;d\mu_2)\|_{L^2(d\mu_2)}^2}\geq 	K_{1}^2+\lambda
\end{align*}
which is nothing but the left side inequality of the expression \eqref{sa prop1 2} and the proof is complete.
\end{proof}
Let $M$ be a fixed positive integer, and assume that $\Phi_M(z;d\mu_2) =\sum_{k=0}^{M} a_{k,M} z^k$ with $a_{M,M}=1$. Now for every $n-1 \geq M$, we define the polynomials $	\Pi_{n,M}(z)$ by
\begin{align*}
			\Pi_{n,M}(z)= n \int_{0}^{z} z^{n-M-1} \Phi_M(z;d\mu_2) dz .
\end{align*}
		Then the polynomials $\{	\Pi_{n,M}(z)\}_{n-1 \geq M}$ satisfy the following properties:
\begin{proposition}\cite{constructive approx2003- berriochoa cachafeiro- Strong asymptotics for sobolev op} \label{Results for the defined polynomial above  }
			The following holds:
\begin{enumerate}
				\item [\rm(1)]$\|\Pi_{n,M}^{\prime}(z)\|^2_{L^2(d\mu_2)}= n^2 \|\Phi_M(z;d\mu_2)\|^2_{L^2(d\mu_2)}.$
				\item [\rm(2)] $\displaystyle\lim_{n \rightarrow \infty}	\|\Pi_{n,M}(z)\|^2_{L^2(d\mu_1)}=\|\Phi_M(z;d\mu_2)\|^2_{L^2(d\mu_1)}$.
				\item [\rm(3)] $\displaystyle\lim_{M\to \infty}\displaystyle\lim_{n \to \infty} \dfrac{\Pi_{n,M}(z)}{z^n} = \dfrac{\overline{D(	\mu^{\prime}_2,\bar{z}^{-1};d\mu_2)^{-1}}}{\overline{D(\mu^{\prime}_2,0;d\mu_2)^{-1}}}$ uniformly on the compact subset $K \in \mathbb{C}\setminus \overline{\mathbb{D}}$.
\end{enumerate}
\end{proposition}
If  $M$ is a fixed positive integer, and if $n-1\geq M $, then $\tilde{\Psi}_n(z)$ can be written as
\begin{align}\label{sa 3.13}
			\tilde{\Psi}_n(z)= \Pi_{n,M}(z)-\sum_{k=0}^{n-1} \beta_{n,k,M} \tilde{\Psi}_k(z).
\end{align}
		Now we will study the coefficients of the above expansion satisfied by the Sobolev orthogonal polynomial $\tilde{\Psi}_n(z)$.
\begin{theorem}\label{coefficients of infinite cut of Psi-bar, s=2}
Let $\{\beta_{n,k,M}\}_{k=0}^{n-1}$ be the coefficients of the recurrence \eqref{sa 3.13}. Then, given $\epsilon >0$, there exists a natural number $M_0$ such that for every $M \geq M_0$, there exists a natural number $n_M$ such that for all $n \geq n_M$, we have
\begin{align}\label{sa thm3.1 3.14}
				\dfrac{1}{n^3}\left[	\left|\beta_{n,0,M}\right|^2 n^2+\left|\beta_{n,1,M}\right|^2 1^2+\ldots+	\left|\beta_{n,n-1,M}\right|^2(n-1)^2 \right]\leq \epsilon.
\end{align}
\end{theorem}
\begin{proof}
Using \eqref{sa 3.13}, and by computing the Sobolev norm, we get
\begin{align}\label{sa thm3.1 1}
				\notag\|\Pi_{n,M}(z)\|^2_{\frac{\lambda}{n^2}}-\|\tilde{\Psi}_n(z)\|^2_{\frac{\lambda}{n^2}}&=\left|\beta_{n,0,M}\right|^2 \|\tilde{\Psi}_{0}(z)\|^2_{\frac{\lambda}{n^2}}+\left|\beta_{n,1,M}\right|^2 \|\tilde{\Psi}_{1}(z)\|^2_{\frac{\lambda}{n^2}}+\ldots\\&+\left|\beta_{n,n-1,M}\right|^2 \|\tilde{\Psi}_{n-1}(z)\|^2_{\frac{\lambda}{n^2}}.
\end{align}
Now by the minimal norm property of the orthogonal polynomials we have
\begin{align*}
				\dfrac{\lambda}{k^2}\times k^2  m(\mu_2)	\leq\dfrac{\lambda}{k^2}\times k^2 \|\Phi_{k-1}(z;d\mu_2)\|^2_{d\mu_2} \leq \dfrac{\lambda}{k^2}\|\tilde{\Psi}^{\prime}_k(z)\|_{d\mu_{2}}^2 .
\end{align*}
This implies
\begin{align}\label{sa thm3.1  2}
			\left|\beta_{n,k,M}\right|^2\dfrac{\lambda k^2}{n^2} m(\mu_2) \leq 	\left|\beta_{n,k,M}\right|^2 \|\tilde{\Psi}_k(z)\|^2_{\frac{\lambda}{n^2}} \quad \text{for} \quad k=1,2,\ldots,n-1.													
\end{align}
We denote by $C_n(\mu_k)$,  the moment of the measure $\mu_k$; that is, $C_n(\mu_k)= \langle z^n,1 \rangle_{\mu_k}$.	Now using \eqref{sa thm3.1  2} in \eqref{sa thm3.1 1}, we have
\begin{equation*}
\begin{aligned}
				\left|\beta_{n,0,M}\right|^2 C_0(\mu_1)
				&+\left|\beta_{n,1,M}\right|^2\dfrac{\lambda 1^2}{n^2}m(\mu_2)
				+\cdots
				+\left|\beta_{n,n-1,M}\right|^2
				\dfrac{\lambda (n-1)^2}{n^2}m(\mu_2) \\
				&\leq
				\|\Pi_{n,M}(z)\|^2_{\frac{\lambda}{n^2}}
				-\|\tilde{\Psi}_n(z)\|^2_{\frac{\lambda}{n^2}}.
\end{aligned}
\end{equation*}
This means
\newline		
		$
		\displaystyle
		\|\Pi_{n,M}(z)\|^2_{\frac{\lambda}{n^2}}-\|\tilde{\Psi}_n(z)\|^2_{\frac{\lambda}{n^2}}
		$
\begin{align}\label{sa thm3.1 3}
			\geq \hat{K}\left[	\left|\beta_{n,0,M}\right|^2 +\left|\beta_{n,1,M}\right|^2\dfrac{1^2}{n^2} +\ldots+	\left|\beta_{n,n-1,M}\right|^2\dfrac{(n-1)^2}{n^2} \right],
\end{align}
where $\hat{K}= \min\{C_0(\mu_1), \lambda m(\mu_2)\}$.
		Since $\lambda m(\mu_2) \leq \|\tilde{\Psi}_n(z)\|^2_{\frac{\lambda}{n^2}}$,  \eqref{sa thm3.1 3} implies
\begin{align}\label{sa thm3.1 4}
			\hat{K}\left[
			\left|\beta_{n,0,M}\right|^2
			+\left|\beta_{n,1,M}\right|^2\dfrac{1^2}{n^2}
			+\ldots+
			\left|\beta_{n,n-1,M}\right|^2\dfrac{(n-1)^2}{n^2}
			\notag	\right]
			\\ \qquad \leq
			\|\Pi_{n,M}(z)\|^2_{d\mu_1}
			+\lambda\|\Phi_M(z;d\mu_2)\|^2_{d\mu_2}
			-\lambda m(\mu_2).
\end{align}
Taking into account that  $\lim_{M\rightarrow \infty} \|\Phi_M(z;d\mu_2)\|^2_{d\mu_{2}}=m(\mu_{2})$, we can see that, for,  given $\epsilon >0$ and $\epsilon_{1}=\dfrac{\hat{K}\epsilon}{2+\lambda}$, there exists a positive integer $M_0$ such that for all $M\geq M_0$,
\begin{align}\label{sa thm3.1 5}
			\|\Phi_M(z;d\mu_2)\|^2_{d\mu_{2}} < m(\mu_2)+\epsilon_{1}.
\end{align}
Moreover, by using Proposition \ref{Results for the defined polynomial above  }, we have, for  each fixed $M\geq M_0$, there exists a positive integer $n_M$ such that  for $n\geq n_M$,
\begin{align}\label{sa thm3.1 6}
			\| \Pi_{n,M}(z)\|^2_{d\mu_1}<\| \Phi_M(z;d\mu_2)\|^2_{d\mu_1}+\epsilon_{1}.
\end{align}
Using \eqref{sa thm3.1 5} and \eqref{sa thm3.1 6} in \eqref{sa thm3.1 4} gives
\begin{align}\label{sa thm3.1 7}
			\hat{K}\left[	\left|\beta_{n,0,M}\right|^2 +\left|\beta_{n,1,M}\right|^2\dfrac{1^2}{n^2} +\ldots+	\left|\beta_{n,n-1,M}\right|^2\dfrac{(n-1)^2}{n^2} \right]\leq
			\|\Phi_M(z;d\mu_2)\|^2_{d\mu_1}+\lambda \epsilon_{1}+\epsilon_{1}.
\end{align}
Therefore, for $M \geq M_0$ and $n \geq n_M$, we can write
\begin{align}\label{sa thm3.1 8}
			\| \Phi_n(z;d\mu_2)\|^2_{d\mu_1}+\epsilon_{1}+\lambda \epsilon_{1}\leq \epsilon_{1}n(2+\lambda).
\end{align}
		Clearly	\eqref{sa thm3.1 7} and \eqref{sa thm3.1 8} implies
\begin{align*}
			\dfrac{\hat{K}}{n^3}\left[	\left|\beta_{n,0,M}\right|^2 n^2+\left|\beta_{n,1,M}\right|^2 1^2+\ldots+	\left|\beta_{n,n-1,M}\right|^2(n-1)^2 \right]\leq \epsilon_{1}(2+\lambda).
\end{align*}
		which by substituting the value of $\hat{K}$ gives  \eqref{sa thm3.1 3.14}.
\end{proof}
\begin{corollary}\label{Bounds for the partial sum of coefficients, s=2}
		Given  $\epsilon >0$, there exists a natural number $M_0$ such that for every $M \geq M_0$, there exists a natural number $N_M$ such that for all $n \geq N_M$ and for all $R>1$, the following inequalities holds:
\begin{align*}
			\sum_{k=0}^{\floor{\frac{n}{2}}-1}\dfrac{\left|\beta_{n,k,M}\right|}{R^{n-k}} &< \sqrt{\epsilon}n^{3/2} \dfrac{1}{R^n}\left( \dfrac{1}{n^2}+\dfrac{R^2}{R^2-1}\right)^{1/2}, \\
			\sum_{k=\floor{\frac{n}{2}}}^{n-1} \dfrac{\left|\beta_{n,k,M}\right|}{R^{n-k}}&< \dfrac{\sqrt{\epsilon}n^{3/2}}{\floor{\frac{n}{2}}R^n} \dfrac{1}{\sqrt{R^2-1}}.
\end{align*}
\end{corollary}
\begin{proof}
Using Cachy-Schwarz inequality and Theorem \ref{coefficients of infinite cut of Psi-bar, s=2}, we get
\begin{align*}
			\sum_{k=\floor{\frac{n}{2}}}^{n-1} \dfrac{\left|\beta_{n,k,M}\right|}{R^{n-k}}&=\sum_{k=\floor{\frac{n}{2}}}^{n-1} \dfrac{k\left|\beta_{n,k,M}\right|}{kR^{n-k}} \leq \left(\sum_{k=\floor{\frac{n}{2}}}^{n-1}\left|\beta_{n,k,M}\right|^2k^2\right)^{1/2} \left(\sum_{k=\floor{\frac{n}{2}}}^{n-1} \dfrac{1}{k^2R^{2(n-k)}}\right)^{1/2}\\&\leq \left(\left|\beta_{n,0,M}\right|^2 n^2+\sum_{k=1}^{n-1}\left|\beta_{n,k,M}\right|^2 k^2\right)^{1/2}\left(\dfrac{1}{\left(\floor{\frac{n}{2}}\right)^2}\sum_{k=\floor{\frac{n}{2}}}^{n-1}\dfrac{1}{R^{2(n-k)}}\right)^{1/2}\\&< \dfrac{\sqrt{\epsilon}n^{3/2}}{\floor{\frac{n}{2}}R^n}\left(\sum_{k=0}^{\infty} R^{2k}\right)^{1/2}=\dfrac{\sqrt{\epsilon}n^{3/2}}{\floor{\frac{n}{2}}R^n} \dfrac{1}{\sqrt{R^2-1}}.
\end{align*}
Similarly, using Cauchy-Schwarz inequality and Theorem \ref{coefficients of infinite cut of Psi-bar, s=2} gives
\begin{align*}
			\sum_{k=0}^{\floor{\frac{n}{2}}-1}\dfrac{\left|\beta_{n,k,M}\right|}{R^{n-k}}&= \dfrac{n\left|\beta_{n,0,M}\right|}{n R^n}+ \sum_{k=1}^{\floor{\frac{n}{2}}-1}\dfrac{k\left|\beta_{n,k,M}\right|}{k R^{n-k}}\\&\leq \left(n^2\left|\beta_{n,0,M}\right|^2+\sum_{k=1}^{\floor{\frac{n}{2}}-1} k^2\left|\beta_{n,k,M}\right|^2\right)^{1/2}\left(\dfrac{1}{n^2 R^{2n}}+\sum_{k=1}^{\floor{\frac{n}{2}}-1} \dfrac{1}{R^{2(n-k)}k^2}\right)^{1/2}\\&<\left(n^2\left|\beta_{n,0,M}\right|^2+\sum_{k=1}^{n-1} k^2\left|\beta_{n,k,M}\right|^2\right)^{1/2}\left(\dfrac{1}{n^2 R^{2n}}+\dfrac{1}{R^{2n}}\sum_{k=1}^{\infty} R^{2k}\right)^{1/2}\\& <\sqrt{\epsilon}n^{3/2} \dfrac{1}{R^n}\left( \dfrac{1}{n^2}+\dfrac{R^2}{R^2-1}\right)^{1/2}.
\tag*{{\qedhere}}		
\end{align*}
\end{proof}
\begin{remark}\label{Modified bounds for Corollary 3.1}
		Given $R>1$, there exists a natural number $M_0$ such that for every $M \geq M_0$, there exists a natural number $N^{\prime}_M$ such that for all $n \geq N^{\prime}_M$ the following hold:
\begin{enumerate}
			\item [\rm (1)]  $\sum_{k=\floor{\frac{n}{2}}}^{n-1} \dfrac{\left|\beta_{n,k,M}\right|}{R^{n-k}}< \dfrac{1}{3}$,
			\item [\rm (2)] $\sum_{k=0}^{\floor{\frac{n}{2}}-1}\dfrac{\left|\beta_{n,k,M}\right|}{R^{n-k}}<\dfrac{1}{3}$.
\end{enumerate}
\end{remark}
The preceding discussion provides the foundation for the following proposition, which plays a crucial role in establishing the main theorem of the paper.
\begin{proposition}\label{Uniform bound for Psi-bar, s=2}
	The sequence $\left\{\frac{\tilde{\Psi}_n(z)}{z^n}\right\}_{n\geq 1}$ is uniformly bounded on compact subsets of $|z|>1$.
\end{proposition}
\begin{proof}
	The result is obtained by an analogous application of the method used in Theorem 1 of \cite{constructive approx2003- berriochoa cachafeiro- Strong asymptotics for sobolev op}.
\end{proof}
Now, we are in position to prove our main theorem of this section.
\begin{theorem}\label{Strong Asymptotic for Psi-bar,s=2} 
	Let $\left\{\tilde{\Psi}_n(z)\right\}$ be the sequence of monic orthogonal polynomials with respect to the Sobolev inner product \eqref{sa ip s=2 3.1} on the unit circle. Then
\begin{align*}
		\lim_{n \to \infty} \dfrac{\tilde{\Psi}_n(z)}{z^n}=\dfrac{\overline{D(	\mu^{\prime}_2,\bar{z}^{-1};d\mu_2)^{-1}}}{\overline{D(\mu^{\prime}_2,0;d\mu_2)^{-1}}}\quad {\mbox{uniformly on compact subsets of}} \,\,  |z|>1.
	\end{align*}
\end{theorem}
\begin{proof}
	In view of Proposition \ref{Results for the defined polynomial above  }, it is  enough to show that
\begin{align*}
		\lim_{M\to \infty}\lim_{n \to \infty} \left(\dfrac{\Pi_{n,M}(z)}{z^n}-\dfrac{\tilde{\Psi}_n(z)}{z^n}\right)=0, \quad \text{ uniformly on compact subsets of} \quad |z|>1.
\end{align*}
	For the compact subset $\Omega $ of $|z|>1$, let $r= \min_{z\in \Omega} |z|$, and take $R$ such that $1<R<r$. Now from \eqref{sa 3.13}, we can write
\begin{align}\label{sa thm3.2 1}
		\left| \dfrac{\Pi_{n,M}(z)}{z^n}-\dfrac{\tilde{\Psi}_n(z)}{z^n}\right|= \left| \sum_{k=0}^{n-1} \dfrac{\beta_{n,k,M}}{z^{n-k}}\dfrac{\tilde{\Psi}_k(z)}{z^k}\right| \leq \sum_{k=0}^{n-1}\dfrac{\left| \beta_{n,k,M}\right|}{R^{n-k}} \left|\dfrac{\tilde{\Psi}_k(z)}{z^k}\right|.
\end{align}
	By using  Proposition \ref{Uniform bound for Psi-bar, s=2}, we can write \eqref{sa thm3.2 1} as
\begin{align*}
		\left| \dfrac{\Pi_{n,M}(z)}{z^n}-\dfrac{\tilde{\Psi}_n(z)}{z^n}\right|\leq K_3  \sum_{k=0}^{n-1}\dfrac{\left| \beta_{n,k,M}\right|}{R^{n-k}} .
\end{align*}
	Now fix $\delta > 0$. By Corollary \ref{Bounds for the partial sum of coefficients, s=2}, there exists $M_0$ such that, for every $M \geq M_0$, there exists $N_M$ such that, for all $n \geq N_M$,
\begin{align}\label{sa thm3.2 2}
		\sum_{k=0}^{n-1}\dfrac{\left| \beta_{n,k,M}\right|}{R^{n-k}} <	\dfrac{\sqrt{\delta}n^{3/2}}{R^n}\left[\dfrac{1}{\sqrt{R^2-1}}\dfrac{1}{\floor{\frac{n}{2}}}+\left(\dfrac{1}{n^2}+\dfrac{R^2}{R^2-1}\right)^{1/2}\right].
\end{align}
	Since $R>1$ is fixed,  $\dfrac{n^{3/2}}{R^n} \to 0$, and the expression in the square bracket remains bounded as $n \to \infty$. Hence the right hand side of \eqref{sa thm3.2 2} tends to zero. Consequently, for every fixed $M \geq M_0$,  we have
\begin{align*}
		\lim_{n \to \infty}\sup_{z \in \Omega}	\left| \dfrac{\Pi_{n,M}(z)}{z^n}-\dfrac{\tilde{\Psi}_n(z)}{z^n}\right| =0.
\end{align*}
	Thus, for every $M \geq M_0$, the difference converges uniformly to zero on $\Omega$. Therefore,
\begin{align*}
		\lim_{M\to \infty}	\lim_{n \to \infty} \left( \dfrac{\Pi_{n,M}(z)}{z^n}-\dfrac{\tilde{\Psi}_n(z)}{z^n}\right)=0
\end{align*}	
uniformly on $\Omega$.
\end{proof}
Since $\mu_2  \in \mathcal{S}$,  using Theorem \ref{Strong Asymptotic for Psi-bar,s=2} provide the following result immediately.
\begin{corollary}\label{Asymptotic consequences of Psi bar(s=2) and Phi(mu 2)}
		$\lim_{n \to \infty}\dfrac{\tilde{\Psi}_n(z)}{\Phi_n(z;d\mu_2)}=1$
		uniformly on the compact subsets of $|z|>1$.
\end{corollary}
As a consequence of the above results, we obtain the following ratio asymptotics for $\tilde{\Psi}_n(z)$.
\begin{theorem}\label{ratio asymptotic}
Let $\left\{\tilde{\Psi}_n(z)\right\}$ be the sequence of monic orthogonal polynomials with respect to the Sobolev inner product \eqref{sa ip s=2 3.1} on the unit circle. Then
		$\displaystyle\lim_{n \to \infty} \dfrac{\tilde{\Psi}_{n+1}(z)}{\tilde{\Psi}_n(z)} =z $
		uniformly on the compact subsets of $|z|>1$.
\end{theorem}
\begin{proof}
Since
\begin{align*}
			\dfrac{\tilde{\Psi}_{n+1}(z)}{\tilde{\Psi}_n(z)}=  \dfrac{\tilde{\Psi}_{n+1}(z)}{\Phi_{n+1}(z;d\mu_2)}\cdot \dfrac{\Phi_{n+1}(z;d\mu_2)}{\Phi_{n}(z;d\mu_2)} \cdot \dfrac{\Phi_{n}(z;d\mu_2)}{\tilde{\Psi}_n(z)},
\end{align*}
		the result follows from Corollary \ref{Asymptotic consequences of Psi bar(s=2) and Phi(mu 2)} and the fact that $\mu_2 \in \mathcal{N}$.
\end{proof}
If we proceed similarly as discussed in Section $2$, since the limit function in the strong asymptotic formula is non vanishing and analytic on $|z|>1$, by the application of the Hurwitz's theorem we can conclude that all the zeros of  $\tilde{\Psi}_n(z)$, lies in $|z|<1+\delta$.  This leads to the following statement :
\begin{corollary}\label{application of Hurwitz theorem for s=2}
For every $\delta >0$ and $n$ sufficiently large, the zeros of $\tilde{\Psi}_{n}(z)$ are in $|z|<1+\delta$.
\end{corollary}
\begin{remark}\label{ about strong asymptotics for s=2 case }
Contrast to the observation in   Remark \ref{observation on SA, for s=2}, in case of $\tilde{\Psi}_n(z)$, the limit of the strong asymptotic formula in Theorem \ref{Strong Asymptotic for Psi-bar,s=2} depends on the measure $\mu_{2}$ as exist in the literature \cite{JAT1999- Aptekarev Berriochoa Cachafeiro-Strong asymptotics, constructive approx2003- berriochoa cachafeiro- Strong asymptotics for sobolev op, JCAM1998- Martinez-Asymptotics of the sobolev op}. In fact, by Corollary \ref{Asymptotic consequences of Psi bar(s=2) and Phi(mu 2)},   $\tilde{\Psi}_n(z)$ behave like $\Phi_{n}(z; d\mu_2)$ on the compact subsets of $\left|z\right|>1,$ for large $n$.
\end{remark}

The following result is a direct consequence of the strong asymptotic formulas established in this section and is stated without proof.
\begin{theorem}
				Let  $\{\tilde{\Psi}_{n}(z)\}_{n\geq 0}$ be the sequence of monic Sobolev orthogonal polynomials with respect to the inner product given by  \eqref{sa ip s=2 3.1}.  Let $K$ be the compact subset of $|z|>1$. Then the following holds:
\begin{enumerate}
					\item[\rm(1)] $\lim_{n \to \infty} \left( \sup_{z\in K} |\tilde{\Psi}_n(z)|\right)^{1/n} = \max_{z\in K} |z|$.
					\item[\rm(2)] $\lim_{n \to \infty} \left( \inf_{z\in K} |\tilde{\Psi}_n(z)|\right)^{1/n} = \min_{z\in K} |z|$.
\end{enumerate}
\end{theorem}

\section{Illustration with specific measures}
	
In the sequel, we support our analysis by considering specific case for each of the measure $\mu_1$ and $\mu_2$. For this purpose, we construct a polynomial with respect the inner product of the form
\begin{align}\label{sa 4.1}
		\langle f,g\rangle _{\frac{\lambda}{n^s}}:= \int_{\mathbb{T}} f(z)\overline{g(z)} \dfrac{1-\left|\alpha\right|^2}{\left|z-\alpha\right|^2}\dfrac{dz}{ 2\pi \iota z}+\dfrac{\lambda}{n^s} \int_{\mathbb{T}} f^{\prime}(z)\overline{g^{\prime}(z)}\dfrac{\left|z-\beta\right|^2}{1+\left|\beta\right|^2}\dfrac{dz}{ 2\pi \iota z},
\end{align}
	where $\lambda>0$,  $s$ fixed and $s \geq 2$. Here $|\alpha|<1$ and $\alpha \neq 0$ and $|\beta|>1$. We call this polynomial  as the Bernstein-Chritoffel Sobolev orthogonal polynomials.
	The properties of those measures associated with the inner product \eqref{sa 4.1} are well known in the literature; for example, see \cite{Adv. comput. Math2007- cachafeiro marcellan perez-Op with respect to the sum of two measure,Szego-book}. We now denote $\{\Phi_n(z;d\mu_i)\}$ as the sequence of monic orthogonal polynomials with respect to the measure $\mu_i$ for $i=1,2$. Then we have the following important asymptotic formulas  stated below:
\begin{itemize}
		\item Since $\mu_1\in \mathcal{S}$, by Szeg\"o strong asymptotic formula, we have
			\begin{align}\label{sa 4.2}
				\lim_{n \to \infty} \dfrac{\Phi_n(z,d\mu_1)}{z^n} =1-\dfrac{\alpha}{z}, \quad \left|\alpha\right|<1, \quad \alpha \neq 0
			\end{align}
			uniformly on compact subsets of $|z|>1$.
		\item Similarly
\begin{equation}\label{sa 4.3}
				\lim_{n \to \infty} \dfrac{\Phi_n(z,d\mu_2)}{z^n} =\dfrac{z\overline{\beta}}{z\overline{\beta} -1}, \qquad \left| \beta\right|>1
\end{equation}
			uniformly on compact subsets of   $\mathbb{C}\setminus\overline{\mathbb{D}}$.
\end{itemize}
Now to discuss further properties of the Bernstein-Christoffel polynomials, we start with  a connection formula stated below.
\begin{proposition}\label{connection formula}
			Let $\{R_{n}(s; z)\}_{n\geq 0}$ be the monic orthogonal polynomial sequence with respect to the inner product \eqref{sa 4.1}. Then we can write $\Phi_{n}(z,d\mu_1)$ as
\begin{align}\label{sa   4.4}
				\Phi_{n}(z,d\mu_1)=R_{n}(s; z)+\lambda_{n,n-1}R_{n-1}(s; z)+\lambda_{n,n-2}R_{n-2}(s; z),
\end{align}
			where $R_0(z)=1,$ $R_1(z)= z-\alpha$ and for $n \geq 2$
\begin{align*}
\lambda_{n,n-1}&=\dfrac{\lambda}{n^s}\left[-\dfrac{n(n-1)\beta}{1+\left|\beta\right|^2}-\alpha(n-1)^2+\dfrac{\alpha \beta(n-1)}{1+\left|\beta\right|^2}\dfrac{R_{n-1}^{(n-2)}(s; 0)}{(n-3)!}\right]\dfrac{1}{\lvert|R_{n-1}(s; z)\rvert|^2_{\frac{\lambda}{n^s}}},\\ \lambda_{n,n-2}&= \dfrac{\lambda}{n^s}\dfrac{\alpha\beta(n-1)(n-2)}{1+\left|\beta\right|^2}\dfrac{1}{\lvert| R_{n-2}(s; z)\rvert|^2_{\frac{\lambda}{n^s}}}.
\end{align*}
\end{proposition}
\begin{proof}
We can write the monic orthogonal polynomai $R_n(s; z)$ as
\begin{align*}
				R_n(s; z)=z^n + \sum_{i=1}^{n} \dfrac{R_n^{(n-i)}(s; 0)}{(n-i)!}z^{n-i},
\end{align*}
and its derivative polynomial $R^{\prime}_n(s; z)$ as
\begin{align*}
				R_n^{\prime}(s; z)=nz^{n-1} + \sum_{i=1}^{n-1} \dfrac{R_n^{(n-i)}(s; 0)}{(n-i-1)!}z^{n-i-1}.
\end{align*}
Now we can write
\begin{align*}
				\Phi_{n}(z,d\mu_1)=R_{n}(s; z)+\sum_{j=0}^{n-1}\lambda_{n,j} R_{j}(s; z),
\end{align*}
where
\begin{align*}
\lambda_{n,j}= \dfrac{\langle \Phi_{n}(z,d\mu_1),R_{j}(s; z) \rangle_{\frac{\lambda}{n^s}}}{{\langle R_{j}(s; z) ,R_{j}(s; z) \rangle}_{\frac{\lambda}{n^s}}}, \quad j=0,1,\ldots,n-1.
\end{align*}
Now consider
\begin{align*}
				I=&\langle \Phi_{n}(z,d\mu_1),R_{j}(s; z) \rangle_{\frac{\lambda}{n^s}}\\=& \int_{\mathbb{T}}\Phi_{n}(z,d\mu_1)\overline{R_{j}(s; z)}\dfrac{1-\left|\alpha\right|^2}{\left|z-\alpha\right|^2}\dfrac{dz}{ 2\pi \iota z}+\dfrac{\lambda}{n^s}\int_{\mathbb{T}}\Phi^{\prime}_{n}(z,d\mu_1)\overline{R^{\prime}_{j}(s; z)}\dfrac{\left|z-\beta\right|^2}{1+\left|\beta\right|^2}\dfrac{dz}{ 2\pi \iota z}.
\end{align*}
Using the orthogonality of $\Phi_{n}(z,d\mu_1)$, the first term vanishes. So we have
\begin{align*}
				I&=\dfrac{\lambda}{n^s}\int_{\mathbb{T}}\Phi^{\prime}_{n}(z,d\mu_1)\overline{R^{\prime}_{j}(s; z)}\dfrac{\left|z-\beta\right|^2}{1+\left|\beta\right|^2}\dfrac{dz}{ 2\pi \iota z}\\
				&=
				\dfrac{\lambda}{n^s}\int_{\mathbb{T}}\biggl( njz^{n-j}-\alpha(n-1)jz^{n-j-1}
				+n\sum_{i=1}^{j-1}\dfrac{\overline{R_{j}^{(j-i)}(s; 0)}}{{(j-i-1)!}} z^{n-j+i}\\& \qquad -\alpha(n-1)\sum_{i=1}^{j-1}  \dfrac{\overline{R_{j}^{(j-i)}(s; 0)}}{{(j-i-1)!}} z^{n-j+i-1} \biggr) \dfrac{\left|z-\beta\right|^2}{1+\left|\beta\right|^2}\dfrac{dz}{ 2\pi \iota z}\\
				&=\dfrac{\lambda}{n^s}\left(I_1+I_2+I_3+I_4\right),
\end{align*}
where
\begin{align*}
				I_1&= nj \int_{\mathbb{T}}z^{n-j}\dfrac{\left|z-\beta\right|^2}{1+\left|\beta\right|^2}\dfrac{dz}{ 2\pi \iota z},\\  I_2 &=-\alpha (n-1)j\int_{\mathbb{T}}z^{n-j-1} \dfrac{\left|z-\beta\right|^2}{1+\left|\beta\right|^2}\dfrac{dz}{ 2\pi \iota z}, \\I_3&= n\sum_{i=1}^{j-1}\dfrac{\overline{R_{j}^{(j-i)}(s; 0)}}{{(j-i-1)!}} \int_{\mathbb{T}}z^{n-j+i}\dfrac{\left|z-\beta\right|^2}{1+\left|\beta\right|^2}\dfrac{dz}{ 2\pi \iota z} , \\ I_4&= -\alpha(n-1)\sum_{i=1}^{j-1} \dfrac{\overline{R_{j}^{(j-i)}(s; 0)}}{{(j-i-1)!}} \int_{\mathbb{T}}z^{n-j+i-1}\dfrac{\left|z-\beta\right|^2}{1+\left|\beta\right|^2}\dfrac{dz}{ 2\pi \iota z}.
\end{align*}
By an easy computation we can write $I_1 = I_1^{(1)}+I_1^{(2)}+I_1^{(3)},$ where
\begin{align*}
				I_1^{(1)}&=nj\int_{\mathbb{T}}z^{n-j} \dfrac{dz}{2\pi\iota z},\\
				I_1^{(2)}&= -\dfrac{nj\overline{\beta}}{1+\left|\beta\right|^2}\int_{\mathbb{T}} z^{n-j+1}\dfrac{dz}{2\pi\iota z},\\
				I_1^{(3)}&= -\dfrac{nj\beta}{1+\left|\beta\right|^2}\int_{\mathbb{T}} z^{n-j-1}\dfrac{dz}{2\pi\iota z}
\end{align*}
and $0 \leq j \leq n-1.$
Since $1\leq n-j \leq n$, by using Cauchy integral formula  we conclude that $I_1^{(1)}=0.$ Similarly $2\leq(n+1)-j\leq (n+1)$ implies $I_1^{(2)}=0$. Note that  $ 0\leq (n-1)-j\leq (n-1)$ implies $I_1^{(3)}$ is non-zero for $j=n-1$. Hence
\begin{align*}
				I_1^{(3)}=-\dfrac{n(n-1)\beta}{1+\left|\beta\right|^2}.
\end{align*}
		
Similarly, we can write $I_2= I_2^{(1)}+I_2^{(2)}+I_2^{(3)},$ where
\begin{align*}
				I_2^{(1)}&= -\alpha (n-1)j\int_{\mathbb{T}}z^{n-j-1}\dfrac{dz}{ 2\pi \iota z} ,\\
				I_2^{(2)}&= \dfrac{\alpha \beta (n-1)j}{1+\left|\beta\right|^2}\int_{\mathbb{T}}z^{n-j-2}\dfrac{dz}{ 2\pi \iota z},\\
				I_2^{(3)}&= \dfrac{\alpha \overline{\beta} (n-1)j}{1+\left|\beta\right|^2}\int_{\mathbb{T}}z^{n-j}\dfrac{dz}{ 2\pi \iota z}
\end{align*}
and $0\leq j \leq n-1$. For $I_2^{(1)}$, note that as  $ 0\leq n-j-1\leq (n-1)$, so only for $j=n-1$ the above integral exists and by using Cauchy integral formula, it reduces to $I_2^{(1)}=-\alpha (n-1)^2.$
Again, $-1\leq n-j-2\leq (n-2)$ guarantees that the integral $I_2^{(2)}$ exists only for $j=n-2$. By using Cauchy integral formula we get
\begin{align*}
				I_2^{(2)}=\dfrac{\alpha\beta(n-1)(n-2)}{1+\left|\beta\right|^2}.
\end{align*}
For $1\leq n-j \leq n$, we conclude that $I_2^{(3)}=0$.	In a similar manner,  we can write $	I_3 = I_3^{(1)}+I_3^{(2)}+I_3^{(3)},$ where
\begin{align*}
				I_3^{(1)}&= n\sum_{i=1}^{j-1}\dfrac{\overline{R_{j}^{(j-i)}(s; 0)}}{{(j-i-1)!}}\int_{\mathbb{T}} z^{n-j+i}\dfrac{dz}{2\pi\iota z},\\
				I_3^{(2)}&=-\dfrac{n\beta}{1+\left|\beta\right|^2}\sum_{i=1}^{j-1}\dfrac{\overline{R_{j}^{(j-i)}(s; 0)}}{{(j-i-1)!}}\int_{\mathbb{T}} z^{n-j+i-1}\dfrac{dz}{2\pi\iota z},\\
				I_3^{(3)}&=-\dfrac{n\overline{\beta}}{1+\left|\beta\right|^2}\sum_{i=1}^{j-1}\dfrac{\overline{R_{j}^{(j-i)}(s; 0)}}{{(j-i-1)!}}\int_{\mathbb{T}} z^{n-j+i+1}\dfrac{dz}{2\pi\iota z}.
\end{align*}
$1\leq i\leq j-1\leq n-2 $ implies $2\leq i-j+n \leq 2n-2.$
So $I_3^{(1)}=0$ by Cauchy integral formula.
From above inequality we get $ 1\leq i-j+n-1\leq 2n-3$. This implies $	I_3^{(2)}=0$. Similarly, $ 3\leq i-j+n+1\leq 2n-1$ implies I$_3^{(3)}=0.$
We now write the integral $I_4$ as $I_4 = I_4^{(1)}+I_4^{(2)}+I_4^{(3)},$	where
\begin{align*}
					I_4^{(1)} &=	-\alpha(n-1)\sum_{i=1}^{j-1} \dfrac{\overline{R_{j}^{(j-i)}(s; 0)}}{{(j-i-1)!}}\int_{\mathbb{T}} z^{n-j+i-1}\dfrac{dz}{ 2\pi \iota z},\\
					I_4^{(2)}&=\dfrac{\alpha\beta(n-1)}{1+\left|\beta\right|^2}\sum_{i=1}^{j-1} \dfrac{\overline{R_{j}^{(j-i)}(s; 0)}}{{(j-i-1)!}}\int_{\mathbb{T}} z^{n-j+i-2}\dfrac{dz}{ 2\pi \iota z} \qquad {\mbox{and}}\\
					I_4^{(3)}&=\dfrac{\alpha\overline{\beta}(n-1)}{1+\left|\beta\right|^2}\sum_{i=1}^{j-1} \dfrac{\overline{R_{j}^{(j-i)}(s; 0)}}{{(j-i-1)!}}\int_{\mathbb{T}}z^{n-j+i}\dfrac{dz}{ 2\pi \iota z}.
\end{align*}
Since $1\leq n-j+i-1\leq 2n-3$,  we get $I_4^{(1)}=0.$ Proceeding similarly as before, $2\leq n+i-j\leq 2n-2$ implies $0\leq n+i-j-2\leq 2n-4 $. So by Cauchy integral formula, $I_4^{(2)}$ is non zero if $n+i-j=2$. But $0\leq j\leq n-1$ implies $-(n-2)\leq i \leq 1.$ So for $i=1, j=n-1$ the above integral exists. This gives
\begin{align*}
					I_4^{(2)}=\dfrac{\alpha\beta(n-1)}{1+\left|\beta\right|^2} \dfrac{R_{n-1}^{(n-2)}(s; 0)}{(n-3)!}.
\end{align*}
				
Since $2\leq i-j+n \leq 2n-2$, again by using the same arguments as above we can conclude that $I_4^{(3)}=0.$ Combining all the above values gives the result.
\end{proof}
The following observation is immediate.
\begin{remark}\label{asymptotic behaviour of the connection coefficients }
				It is easy to see that for $s>2$, $\lambda_{n,n-1}\to 0$ and $\lambda_{n,n-2} \to 0 $ as $ n \to \infty$.
\end{remark}
\begin{remark}
Lemma \ref{Uniform convergence of ratio of Psi and Phi, s>2} in Section 2 was proved in a general setting. In the present case, concerning the Bernstein--Christoffel Sobolev orthogonal polynomials associated with the inner product \eqref{sa 4.1}, the same result can be obtained from \eqref{sa 4.4}, although by a different argument. For completeness, we give this alternative derivation for the case $s>2$ below. We can write \eqref{sa   4.4} as
\begin{align*}
				1=&\dfrac{R_n(s>2; z)}{\Phi_n(z;d\mu_1)}+\lambda_{n,n-1}\dfrac{R_{n-1}(s>2; z)}{\Phi_{n-1}(z;d\mu_1)}\cdot\dfrac{\Phi_{n-1}(z;d\mu_
					1)}{\Phi_n(z;d\mu_1)}\\
                \qquad +&\lambda_{n,n-2} \dfrac{R_{n-2}(s>2; z)}{\Phi_{n-2}(z;d\mu_1)} \cdot\dfrac{\Phi_{n-2}(z;d\mu_1
					)}{\Phi_{n-1}(z;d\mu_1)}\cdot \dfrac{\Phi_{n-1}(z;d\mu_1
					)}{\Phi_{n}(z;d\mu_1)}.
\end{align*}
This can be rewritten as
\begin{align}\label{New recurrence after division}
				1= U_n(z)+\xi_n(z) U_{n-1}(z)+\eta_n(z) U_{n-2}(z)
\end{align}
			where	$U_n(z)= \dfrac{R_n(s>2; z)}{\Phi_n(z;d\mu_1)}$,
\begin{align*}
\xi_n(z) = \lambda_{n,n-1} \dfrac{\Phi_{n-1}(z;d\mu_1)}{\Phi_n(z;d\mu_1)} \quad {\mbox{and}}\quad
\eta_n(z)= \lambda_{n,n-2}\dfrac{\Phi_{n-2}(z;d\mu_1)}{\Phi_{n-1}(z;d\mu_1)}\cdot \dfrac{\Phi_{n-1}(z;d\mu_1)}{\Phi_{n}(z;d\mu_1)}.
\end{align*}
Since $\mu_1 \in \mathcal{N}$, we have $\lim_{n \to \infty}\dfrac{\Phi_{n}(z;d\mu_1)}{\Phi_{n-1}(z;d\mu_1)} =z$ uniformly on the compact subsets of $|z|>1$.  Since for $s>2$, $\lambda_{n,n-1}, \lambda_{n,n-2} \to 0$ as $n \to \infty$, we conclude that $\xi_n \to 0$ and  $\eta_n \to 0$  as $n \to \infty$ on the compact subsets $K$ of $ |z|>1$.  . Now, we choose $N $ so large to get
\begin{align*}
				\sup_{z\in K} \left( \left| \xi_n\right| + \left| \eta_n\right|\right) \leq \dfrac{1}{2}, \qquad n\geq N.
\end{align*}
The finitely many functions $U_0, \ldots , U_{N-1}$ are uniformly bounded on $K$. Choose $M \geq 2$ such that
\begin{align*}
				\sup_{z\in K} \left| U_j\right| \leq M, \qquad j=0,\ldots , N-1.
\end{align*}
In particular, if
\begin{align*}
				\sup_{z\in K} \left| U_{n-1}\right| \leq M , \qquad \sup_{z\in K} \left| U_{n-2}\right| \leq M,
\end{align*}
then we get
\begin{align*}
				\sup_{z\in K} \left| U_n\right| \leq 1+ M \sup_{z\in K} \left( \left| \xi_n\right| + \left|\eta_n\right|\right) \leq 1+\dfrac{M}{2} \leq M.
\end{align*}
Thus $U_n$ is uniformly bounded on $K$. Now using \eqref{New recurrence after division}, we can write
\begin{align*}
				\left| U_n -1 \right| \leq \left| \xi_n\right| \left| U_{n-1}\right| + \left| \eta_n\right| \left| U_{n-2}\right|,
\end{align*}
which implies
\begin{align*}
				\sup_{z\in K} \left| U_n -1\right| \leq M \sup_{z\in K} \left( \left| \xi_n\right| + \left| \eta_n\right|\right) \to 0.
\end{align*}
Therefore,
			$\lim_{n \to \infty}	U_n(z) =1$ uniformly on the compact subsets $K$. Since  $K$ is arbitrary, the desired result follows.
\end{remark}
\subsection{Asymptotic Distribution of the zeros}
As an application of the strong asymptotic formula established in Theorem \ref{Strong Asymptote for Psi, s>2 case} and  Theorem \ref{Strong Asymptotic for Psi-bar,s=2}, we now investigate  the asymptotic  distribution of the zeros of the Bernstein- Christoffel Sobolev orthogonal polynomials in the cases  $s>2$ and $s=2$.
Let $t_{n,1}, t_{n,2},\ldots, t_{n,n}$ denote the zeros of $R_n(s; z)$ (counted with multiplicity), with respect to the inner product \eqref{sa 4.1}, for the $s\geq2$ case. We associate the discrete unit measures with equal positive masses at each  zeros of $R_{n}(s; z)$, that is
\begin{align}\label{sa 4.6}
			\nu_n = \dfrac{\delta_{t_{n,1}}+\delta_{t_{n,2}}+\ldots +\delta_{t_{n,n}}}{n}.
\end{align}
Then we have the following result.
\begin{theorem}\label{distribution of zeros for s=2, s>2 case}
\begin{enumerate}[{\rm(1)}]
\item For $s>2$, let $\{z_{n,k}\}_{k=1}^n$ be the zeros (counted with multiplicity) of the monic Sobolev orthogonal polynomial $R_n(s>2; z)$, with respect to the inner product \eqref{sa 4.1} . Then $\nu_n$ defined by \eqref{sa 4.6}, converges to  normalized Lebesgue measure ${\mu_{\mathbb{T}}}$ in the weak* sense.
\item For $s=2$, let $\{\xi_{n,k}\}_{k=1}^n$ be the zeros of the monic Sobolev orthogonal polynomial $R_n(2; z)$, counted with multiplicity.  Then $\nu^{\prime}_n$ defined  similarly as \eqref{sa 4.6}, by using the zeros $\xi_{n,k}$'s, converges to ${\mu_{\mathbb{T}}}$ in the weak* sense.
\end{enumerate}
\end{theorem}
\begin{proof}
We first prove the case $s>2$.
It is easy to see that
\begin{align}\label{sa thm4.1 1}
					\dfrac{\log|R_{n}(s>2; z)|}{n}= \sum_{k=1}^{n} \dfrac{1}{n}\log|z-z_{n,k}|= \int_{z\notin \text{supp\,}{\nu_n}} \log|z-w| d\nu_n.
\end{align}
By Corollary \ref{Hurwitz, zeros of Psi, s>2}, $\text{supp}{\nu_n} \subseteq \overline{\{z : |z|<1+\epsilon\}}$ for every $\epsilon>0$. Again, using Theorem \ref{Strong Asymptote for Psi, s>2 case} and right hand side of \eqref{sa thm4.1 1}, we can write
\begin{align}\label{sa thm4.1 2}
					\log|R_{n}(s>2; z)| = n \log|z| + \log\left|1-\dfrac{\alpha}{z}\right|+ {o}(1)
\end{align}
on the compact subsets $K$ of $|z|>1$. \eqref{sa thm4.1 2} implies
\begin{align}\label{sa thm4.1 3}
					\lim_{n \to \infty} \dfrac{1}{n}\log|R_n(s>2; z)| = \log|z|
\end{align}
on the compact subsets $K$ of $|z|>1$. Since $\{z: |z|>1+\epsilon\} \subseteq \{ z : |z|>1\}$, we can choose the compact set $K$  such that $K \subset \{z: |z|>1+\epsilon\}$. Thus we can write right hand side of \eqref{sa thm4.1 3} as the logarithmic potential of the measure $\nu_n$. That can be represented as
\begin{align*}
					P^{\nu_n}(z)=\int_{K}\log|z-w| d\nu_n= \dfrac{\log|R_{n}(s>2; z)|}{n}.
\end{align*}
From \eqref{sa thm4.1 2}, we have
\begin{align}\label{sa thm4.1 4}
					\lim_{n \to \infty} P^{\nu_n}(z)= \log|z| .
\end{align}
Now logarithmic potential of the normalized Lebesgue measure ${\mu_{\mathbb{T}}}$ on the unit circle is given by
(see \cite[Page 226]{2002-Andrievskii-HPBlattDiscrepancy-Approx-SpringerBook})
\begin{align*}
					P^{{\mu_{\mathbb{T}}}}(z)&= \int_{\mathbb{T}} \log|z-\xi| d{\mu_{\mathbb{T}}}(\xi)= \log|z|
\end{align*}
where $|z|>1$. From \eqref{sa thm4.1 4}, we conclude that
\begin{align*}
					\lim_{n \to \infty} P^{\nu_n}(z)= P^{{\mu_{\mathbb{T}}}}(z)\quad \text{uniformly on the compact subsets $K$ of } \quad |z|>1,
\end{align*}
from which we conclude that $\nu_n \stackrel{\ast}{\rightharpoonup} {\mu_{\mathbb{T}}}$.
				
For the case $s=2$, using \eqref{sa 4.3} and Theorem \ref{Strong Asymptotic for Psi-bar,s=2}, we get
\begin{align*}
					\lim_{n \to \infty} \dfrac{R_n(2; z)}{z^n}=\dfrac{z\overline{\beta}}{z\overline{\beta}-1}\quad \text{uniformly on compact subsets of } \quad |z|>1.
\end{align*}
By  similar argument, we can get the $\text{weak}^{*}$ convergence of the zeros $\{\xi_{n,k} \}_{k=1}^n$ and the proof is complete.
\end{proof}
\subsection{ Computations of the zeros via  Hessenberg matrix}
We are now interested in determining the zeros of the Bernstein-Christoffel   Sobolev orthogonal polynomials with respect to the inner product  \eqref{sa 4.1}, which satisfy the recurrence of the form \eqref{sa   4.4}. To this end, we begin by considering the matrix representation of the multiplication by the $z$ operator with respect to the basis $\{ \Phi_n(z;d\mu_1)\}_{n\geq 0}$ is given by
\begin{align}\label{sa 4.11}
			z\Phi(z;d\mu_1) = H^{(\mu_1)}_{\Phi}\Phi(z;d\mu_1),
\end{align}
where $\Phi(z;d\mu_1) = \big[\Phi_{0}(z;d\mu_1), \Phi_{1}(z;d\mu_1), \ldots\big]^{T}$ and $H^{(\mu_1)}_{\Phi}$ is a semi-infinite Hessenberg matrix with $1$'s on the upper diagonal and whose remaining entries are given in terms of the Verblunsky's coefficients $\{\Phi_{n}(0;d\mu_1)\}$ (see \cite{2012- Numerical Algorithm -Castillo Garza Paco, OPUCVol1-Simon Barry}). It is well known that the zeros of the $n$-th orthogonal polynomial $\Phi_{n}(z;d\mu_1)$  are the eigenvalues of the leading principal submatrix $n\times n$ of the Hessenberg matrix $H^{(\mu_1)}_{\Phi}$ (see \cite{OPUCVol1-Simon Barry}), which we will denote by $H^{(\mu_1)}\left(n; \Phi \right)$. In particular, if we use $\{\Phi_{0}(z;d\mu_1), \Phi_{1}(z;d\mu_1), \ldots, \Phi_{n-1}(z;d\mu_1)\}$ as a basis  then we can write \eqref{sa 4.11} as
\begin{align}\label{sa 4.12}
			z\Phi(n; z;d\mu_1) = H^{(\mu_1)}\left(n;\Phi\right)\Phi (n; z;d\mu_1)+\Phi_n(z;d\mu_1)	\begin{bmatrix}
				0\\
				0 \\
				\vdots \\
				1\\
			\end{bmatrix}
\end{align}
where $\Phi(n; z;d\mu_1) = \big[\Phi_{0}(z;d\mu_1), \Phi_{1}(z;d\mu_1), \ldots, \Phi_{n-1}(z;d\mu_1)\big]^{T}$. Let $\{z_{n,j}\}_{j=1}^n$ be the  zeros of the polynomial $\Phi_n(z;d\mu_1)$, which are also the eigenvalues of the Hessenberg matrix $H^{(\mu_1)} \left(n; \Phi\right)$. Thus from \eqref{sa 4.12}, one has $z_{n,j}\Phi(n; z_{n,j};d\mu_1) = H^{(\mu_1)}(n; \Phi)\Phi(n; z_{n,j};d\mu_1), \newline \quad j=1,2\ldots,n.$ Denote $\mathcal{R}(n; s; z) = \big[R_{0}(s; z), R_{1}(s; z), \ldots, R_{n-1}(s; z)\big]^{T}$ where $R_n(s; z)$ is the $n$th  Sobolev orthogonal polynomial with respect to the inner product \eqref{sa 4.1}, where $s$ is fixed, and $s\geq 2$. Then we can write \eqref{sa 4.4} in the following matrix form
\begin{equation*}
			\begin{bmatrix}
				\Phi_{0}(z;d\mu_1) \\
				\Phi_{1}(z;d\mu_1) \\
				\Phi_{2}(z; d\mu_1)\\
				\vdots \\
				\Phi_{n-1}(z;d\mu_1) \\
			\end{bmatrix}
			= \begin{bmatrix}
				1    & 0 & 0 &0 & \dots & 0 \\
				\lambda_{1,0}       & 1 & 0 & 0 & \dots & 0  \\
				\lambda_{2,0} & \lambda_{2,1} & 1 & 0 & \ldots & 0 \\
				\vdots & \vdots& \vdots & \vdots& \ddots  &\vdots\\
				0      & 0 & 0 & \dots &\lambda_{n,n-2} & 1
			\end{bmatrix}
			\begin{bmatrix}
				R_{0}(z) \\
				R_{1}(z) \\
				R_{3}(z) \\
				\vdots\\
				R_{n-1}(z) \\
			\end{bmatrix} .
\end{equation*}
This means
\begin{align}\label{sa 4.13}
			\Phi(n; z;d\mu_1)= T(n; \lambda)\mathcal{R}(n;s; z),
\end{align}
where
\begin{equation*}
			T(n; \lambda)	= \begin{bmatrix}
				1    & 0 & 0 &0 & \dots & 0 \\
				\lambda_{1,0}       & 1 & 0 & 0 & \dots & 0  \\
				\lambda_{2,0} & \lambda_{2,1} & 1 & 0 & \ldots & 0 \\
				\vdots & \vdots& \vdots & \vdots& \ddots  &\vdots\\
				0      & 0 & 0 & \dots &\lambda_{n,n-2} & 1
			\end{bmatrix},
\end{equation*}
and $\Phi(n; z;d\mu_1), \mathcal{R}(n; s; z)$ are defined as above.  Note that $T (n;\lambda)$ is a lower triangular matrix, with all of its diagonal entries as $1$. So $T(n; \lambda)$ is invertible. Thus we can write \eqref{sa 4.13} as
\begin{align}\label{sa 4.14}
			\mathcal{R}(n; s; z)= T(n; \lambda)^{-1}\Phi(n; z;d\mu_1).
\end{align}
Note that, we can also write $R_n(s; z)$ as
\begin{align}\label{fourier expansion of Rn w.r.to Phi_n}
			R_n(s; z)=\Phi_{n}(z;d\mu_1)+\sum_{k=0}^{n-1}d_{n,k} \Phi_k(z;d\mu_1).
\end{align}
\begin{theorem}\label{Computation of zeros by using Hessenberg matrix}
			For $n \geq 2,$ let $B_n$ be the $n \times n$ matrix
\begin{equation}\label{matrix Bn}
				B_n := \begin{bmatrix}
					0  & 0 & 0 &0 & \dots & 0 \\
					0       &  0 & 0& 0 & \dots & 0  \\
					
					\vdots & \vdots& \vdots & \vdots& \ddots  &\vdots\\
					0 & 0 & 0 &0& \ldots & 0 \\
					d_{n,0}   & d_{n,1}& d_{n,2} & \dots &d_{n,n-2} & d_{n,n-1}
				\end{bmatrix},
\end{equation}
where the elements $d_{n,k}$ are as in \eqref{fourier expansion of Rn w.r.to Phi_n}.
Then the zeros of  $R_n(s; z)$ are the solution of the eigenvalue problem of the matrix $H^{(\mu_1)}(n; \Phi)-B_n $.
\end{theorem}
\begin{proof}
			Let
\begin{align}\label{sa thm 4 4.17}
				zR_{n-1}(s; z)= R_n(s; z)+\sum_{k=0}^{n-1}c_{n,k} R_k(s; z), \quad n\geq 1.
\end{align}
Then, we can write \eqref{sa thm 4 4.17} in the form
\begin{align}\label{sa thm 4 4.18}
				z\mathcal{R}(n; s; z)= M(n)\mathcal{R}(n; s; z)+R_n(s; z)\begin{bmatrix}
					0\\
					0 \\
					\vdots \\
					1\\
				\end{bmatrix}
\end{align}
where
\begin{equation*}
				M(n)= \begin{bmatrix}
					c_{1,0}   & 1 & 0 &0 & \dots & 0 \\
					c_{2,0}       & c_{2,1} & 1 & 0 & \dots & 0  \\
					c_{3,0} & c_{3,1} & c_{3,2}& 1 & \ldots & 0 \\
					\vdots & \vdots& \vdots & \vdots& \ddots  &\vdots\\
					c_{n,0}    & c_{n,1}& c_{n,2} & \dots &c_{n,n-2} & c_{n,n-1}
				\end{bmatrix}.
\end{equation*}
Note that the zeros of $R_n(s; z)$ are the eigenvalues of $M(n)$. Denoting $e_n = [0,0,\ldots,1]^T$, from \eqref{sa 4.13} we get
\begin{align}\label{sa thm 4 4.19}
				z \Phi(n; z;d\mu_1)= T(n; \lambda)z\mathcal{R}(n; s; z).
\end{align}
Now employing \eqref{sa 4.12} and \eqref{sa thm 4 4.18} in \eqref{sa thm 4 4.19}, we get
\begin{align}\label{sa thm 4 4.20}
				H^{(\mu_1)}(n; \Phi)\Phi(n; z;d\mu_1)+\Phi_n(z;d\mu_1) e_n = T(n;\lambda)\big[M(n)\mathcal{R}(n; s; z)+R_n(s; z) e_n \big].
\end{align}
Using \eqref{sa 4.14} in \eqref{sa thm 4 4.20}, and the fact that $T(n; \lambda)e_n =e_n$, we have
\newline
		$
		\displaystyle
		H^{(\mu_1)}(n; \Phi)\Phi(n; z;d\mu_1)=
		$
\begin{align}\label{sa thm 4 4.21}
			T(n; \lambda)M(n)T(n; \lambda)^{-1}\Phi(n; z;d\mu_1)+ \big[ R_n(s; z)-\Phi_n(z;d\mu_1)\big]e_n.
\end{align}
On the other hand, by using \eqref{fourier expansion of Rn w.r.to Phi_n}, we can write \eqref{sa thm 4 4.21} as
\begin{align*}
			\big[H^{(\mu_1)}(n; \Phi)-B_n -T(n; \lambda)M(n) T(n; \lambda)^{-1}\big]\Phi(n; z ;d\mu_1)=0
\end{align*}
for every $z$. Define $A := H^{(\mu_1)}(n; \Phi)-B_n -T(n; \lambda)M(n) T(n; \lambda)^{-1}.$ Then from above we can write
\begin{align*}
			A\Phi(n; z ;d\mu_1) =0 \qquad \text{for every} \quad z.
\end{align*}
Write $i$-th row of $A$ as $\left(a_{i,0}, a_{i,1}, \ldots, a_{i,n-1}\right).$ Then the $i$-th component of $	A\Phi(n; z ;d\mu_1) =0$ gives
\begin{align*}
			a_{i,0} \Phi_{0}(z; d\mu_1)+ a_{i,1} \Phi_{1}(z;d\mu_1)+\ldots + a_{i,n-1}\Phi_{n-1}(z;d\mu_1)=0.
\end{align*}
Since the polynomials $\Phi_{0}(z; d\mu_1), \Phi_{1}(z;d\mu_1), \ldots, \Phi_{n-1}(z;d\mu_1) $ are linearly independent, we must have $	a_{i,0}=  a_{i,1} = \ldots = a_{i,n-1}=0.$ This holds for every $i$, and hence $A=0$. Therefore,
\begin{align*}
			H^{(\mu_1)}(n; \Phi)-B_n = T(n; \lambda)M(n) T(n; \lambda)^{-1}.
\end{align*}
Thus $H^{(\mu_1)}(n; \Phi)-B_n $ is similar to $M(n).$ Hence the eigenvalues of $\left(H^{(\mu_1)}(n; \Phi)-B_n \right)$ and the eigenvalues of $\left( M(n)\right),$ are same up to their algebraic multiplicities. Consequently, the zeros of $R_n(s; z)$ are the precisely the eigenvalues of $	H^{(\mu_1)}(n; \Phi)-B_n .$
\end{proof}
	
\subsection{Location  of the zeros}
Motivated by Theorem \ref{Computation of zeros by using Hessenberg matrix}, we now turn our attention to the theoretical and numerical analysis of  locations of the zeros of the  polynomials $R_n(s; z)$  which are defined recursively by relation \eqref{sa 4.4} in Proposition \ref{connection formula}. Following the parameter choices, we proceed to numerically validate the case $s>2$ by specifically considering $s=3$. For both cases $s=2$ and $s=3$, we construct the polynomials $R_n(s; z)$ up to degree $10$. To proceed further, we start with a simple observation. Using induction hypothesis, one can easily see that, $R_n(s; z)$ can be written as $R_n(s; z)= (z-\alpha) Q_{n-1}(s; z)$, for all $n \geq 1$, where $Q_{n-1}(s; z)$ is a $n-1$ degree polynomial, satisfying the recurrence relation
\begin{align*}
		z^{n-1}= Q_{n-1}(s; z)+ \lambda_{n,n-1} Q_{n-2}(s; z)+\lambda_{n,n-2} Q_{n-3}(s; z) \quad \text{for all}\quad n \geq 1.
\end{align*}
From this recursion, it is observed that $\alpha$ is a common zero of the polynomials $R_n(s; z)$, for all $n \geq 1$. An easy computation shows that for each case, there is a common zero at $0.5$. To discuss  the location of the other zeros of these recursive polynomials, we will prove the following results:
\begin{theorem}\label{Location of zeros of R_n, s>2 case}
		For $s>2$, there exist $N_0$, such that for all $n\geq N_0$, all the zeros of the polynomials $R_n(s>2; z)$, defined recursively by \eqref{sa   4.4},  lies in the $\mathbb{D}$.
\end{theorem}
\begin{proof}
		From \eqref{sa   4.4}, we can write
\begin{align*}
			A_n \leq (1+\left|\alpha\right|)+ \left|\lambda_{n,n-1}\right| A_{n-1}+ \left|\lambda_{n,n-2} \right| A_{n-2},
\end{align*}
where $A_n:= \max_{|z|=1} \left| R_n(s; z) \right|$. As observed in Remark \ref{asymptotic behaviour of the connection coefficients }, for $s>2$, $\lambda_{n,n-1}\to 0$ and  $\lambda_{n,n-2} \to 0$, from which we can write $\left|\lambda_{n,n-1}\right| +\left|\lambda_{n,n-2}\right| \leq \dfrac{C}{n^{\delta}}$, where $\delta >0$ is fixed and $C$ is independent of $n$. Now we claim that $A_n \leq \eta $, for all $n$. Note that $A_{0}=1$ and $A_1 \leq 1+\left| \alpha\right|$. We assume $A_j \leq \eta $, for all $j= 0,1,\ldots, n-1$. Then $A_n \leq C_0 + \dfrac{2C\eta}{n^{\delta}}$, where $C_0 := 1+\left| \alpha \right|$. Since $\dfrac{2C\eta}{n^{\delta}} \to 0$, there exists $N_0$, such that, for all $n \geq N_0$ ,
\begin{align*}
			\dfrac{2C\eta}{n^{\delta}} \leq \dfrac{\eta}{2}.
\end{align*}
This implies $A_n \leq C_0 + \dfrac{\eta}{2}$.  Now, if we choose $\eta \geq 2 C_0$, then for all $n \geq N_0$, we conclude $A_n \leq \eta$.  Now, if we define $\eta \geq  \max \{ A_0, A_1,\ldots, A_{N_{0}-1}, 2C_0\}$, then all the cases are covered explicitly. Hence, the claim is proved. Now we define $E_n(z):=R_n(s>2; z)- \Phi_{n}(z;d\mu_1)$. Since, on $|z|=1$, we have $\left| \Phi_{n}(z;d\mu_1)\right|> 1-|\alpha|:=\epsilon$, and $\left| E_n(z)\right|\leq \dfrac{2C\eta}{n^{\delta}}$, we choose $N_0$, such that $\dfrac{\eta}{2}<\epsilon$ and $\epsilon$ is independent of $n$. Then on $|z|=1$, we have
\begin{align*}
			\left| R_n(s>2; z) - \Phi_{n}(z; d\mu_1)\right| < \left| \Phi_n(z;d\mu_1)\right|.
\end{align*}
By Rouche's theorem,  $R_n(s>2; z)$ and $\Phi_n(z;d\mu_1)$ have the same number of zeros inside $\mathbb{D}$. Hence, all the zeros of $R_n(s>2; z)$ lie in $\mathbb{D}$, for all $n \geq N_0$.
\end{proof}
To obtain the location of the zeros of $R_n(2; z)$, we employ some restrictions on the coefficients of $R_n(2; z)$, $\left|z\right|=1$ and prove the following result.
\begin{theorem}\label{Location of  zeros of R_n, s=2 case}
Let $\{ R_n(2; z)\}_{n\geq 0}$ be the sequence of Sobolev orthogonal polynomials, defined recursively by \eqref{sa   4.4}, for $s=2$ case.  Define, for $|z|=1$,
\begin{align*}
			a_n(z):= \dfrac{R_{n-1}(2; z)}{\Phi_{n}(z;d\mu_1)} \quad  {\mbox{and}} \quad b_n(z):= \dfrac{R_{n-2}(2; z)}{\Phi_n(z;d\mu_1)}.
\end{align*}
		If $\sup_{|z|=1} \left| \lambda_{n,n-1} a_n(z)+ \lambda_{n,n-2} b_n(z)\right| <1$, then all the zeros of $R_n(2; z)$ lie in $\mathbb{D}$.
\end{theorem}
\begin{proof}
		From \eqref{sa   4.4}, we can write
\begin{align*}
			R_n(2; z)= \Phi_{n}(z;d\mu_1) \left(1-\lambda_{n,n-1} a_n(z) -\lambda_{n,n-2} b_n(z)\right).
\end{align*}
Set $f(z)= \Phi_n(z;d\mu_1)$, $g(z):= -\lambda_{n,n-1} R_{n-1}(2; z) -\lambda_{n,n-2} R_{n-2}(2; z)$. Then on $|z|=1$, we have
\begin{align*}
			\dfrac{\left| g(z)\right|}{\left|f(z)\right|}= \left| \lambda_{n,n-1} a_n(z) +\lambda_{n,n-2} b_n(z)\right|.
\end{align*}
By hypothesis, $\left| g(z)\right| < \left| f(z)\right|$ for all $|z|=1$. Hence, by Rouche's theorem, $R_n(2; z)= f(z)+g(z)$ and $\Phi_n(z;d\mu_1)$ have the same number of zeros in $\mathbb{D}$. This justifies the fact that $R_n(2; z)$ has all its zeros in $\mathbb{D}$.
\end{proof}

To exhibit the results numerically, we obtain  $\sup_{|z|=1}|\lambda_{n,n-1}a_n+\lambda_{n,n-2}b_n|$
	for $s=2$, in the following table:
\begin{table}[h!]
		\centering
		\begin{tabular}{|c|c|c|c|c|c|}
			\hline
			$n$  & $1$ &$2$ & $3$ & $4$ & $5$ \\ \hline
			$\sup$ value & -- & $0.325000$ & $0.451441$ & $0.494901$ & $0.511366$ \\ \hline
			$n$ & $6$ & $7$ & $8$ & $9$ & $10$\\ \hline
			$\sup$ value & $0.519920$ & $0.525898$ & $0.530718$ & $0.534789$ & $0.538118$ \\ \hline
		\end{tabular}
		\caption{Rouché sup values showing
			$\sup_{|z|=1}|\lambda_{n,n-1}a_n+\lambda_{n,n-2}b_n|<1$
			for the $s=2$ case.}
\end{table}
Clearly the position of the zeros in Figure \ref{Location of zeros for Rn(z), s>2} and  Figure \ref{Location of zeros for Rn(z),  s=2} show that the zeros lie inside $\mathbb{D}$, supporting Theorem \ref{Location of zeros of R_n, s>2 case} and Theorem \ref{Location of  zeros of R_n, s=2 case}. It is also evident from the graphical illustration that the common zero is at $0.5$.


\begin{minipage}{1.0\linewidth}
		\centering
		\begin{minipage}{0.40\linewidth}
			\begin{figure}[H]
				\includegraphics[width=\linewidth]{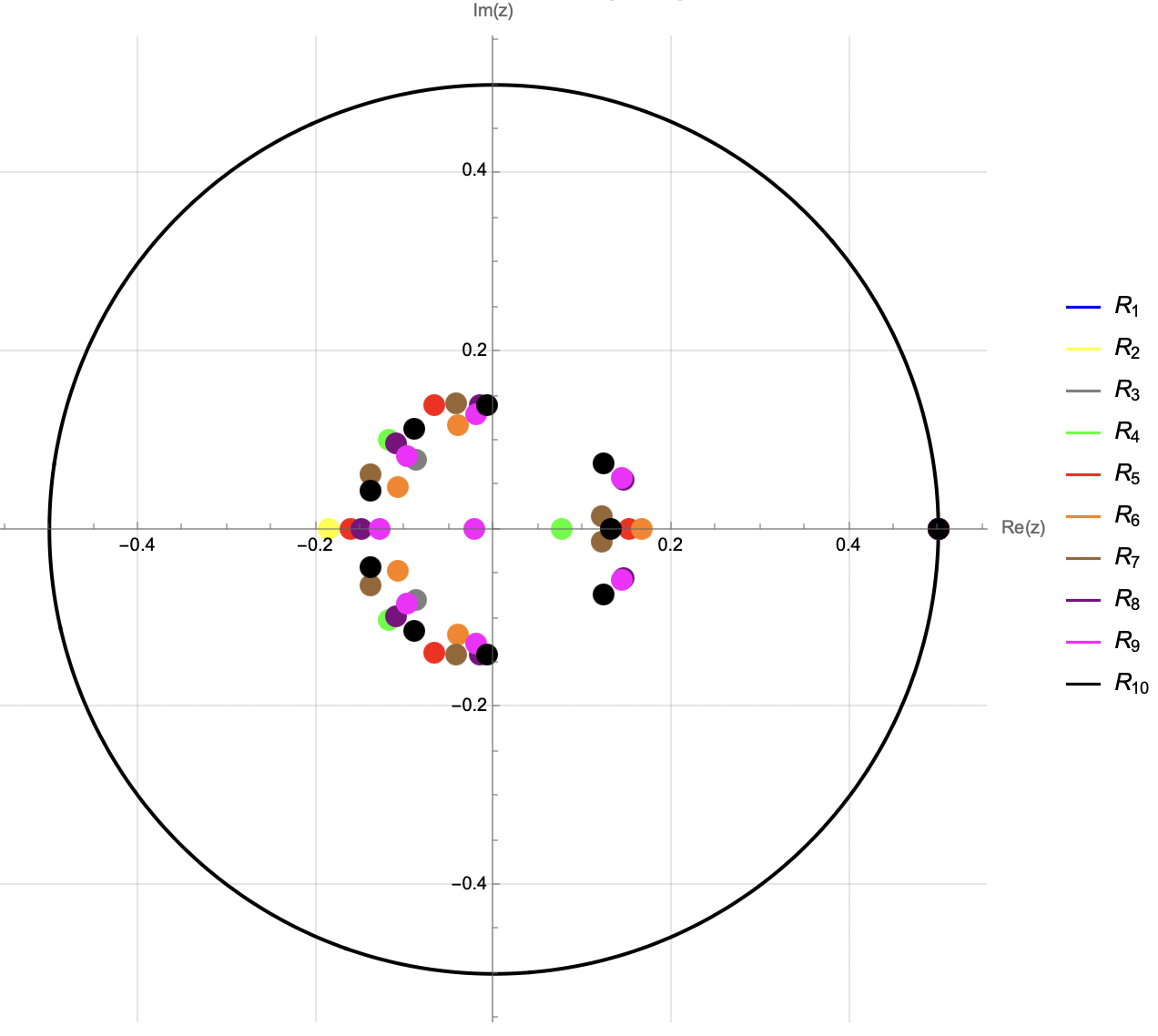}
                \captionof{figure}{Zeros of $R_1(3; z)$, \ldots, $R_{10}(3; z)$}
				\label{Location of zeros for Rn(z), s>2}
			\end{figure}
		\end{minipage}
		\hspace{0.08\linewidth}
		\begin{minipage}{0.40\linewidth}
			\begin{figure}[H]
				\includegraphics[width=\linewidth]{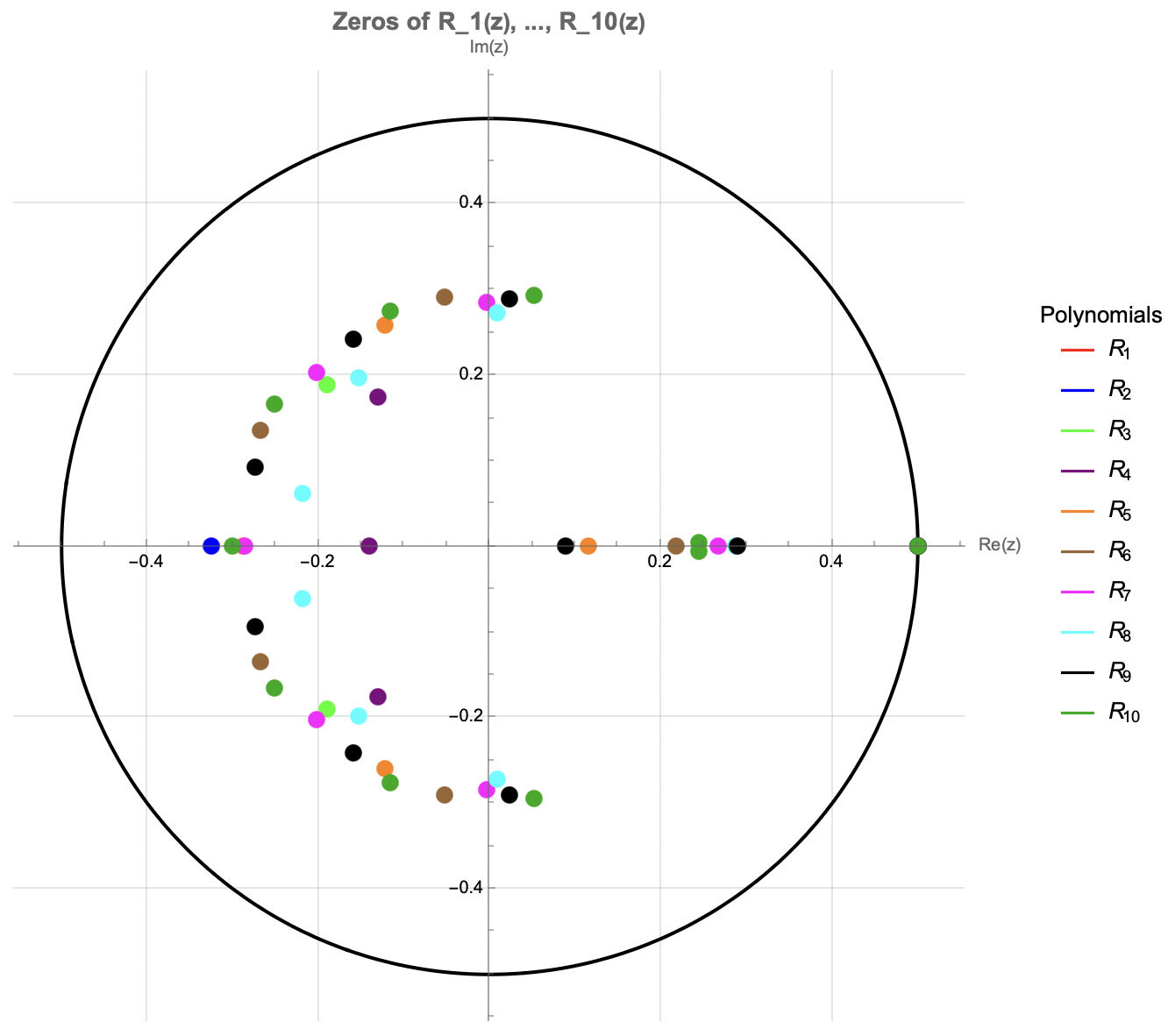}
                \captionof{figure}{Zeros of $R_1(2; z), \ldots, R_{10}(2; z)$}
				\label{Location of zeros for Rn(z),  s=2}
			\end{figure}
		\end{minipage}
\end{minipage}

{${}$}
\newline
{${}$}
We complement the theoretical analysis with a numerical study of the asymptotic behavior established in Sections $2$ and $3$, as well as of the asymptotic relation \eqref{sa thm4.1 3} in Theorem \ref{distribution of zeros for s=2, s>2 case} of Section $4$, for the recursively defined polynomials $R_n(s;z)$ introduced in Proposition \ref{connection formula}.
	Throughout, we fix parameters $\alpha =\frac{1}{2}, \beta = 2, \lambda =1 $. To analyze the $s>2$ case, we choose a particular value $s=3$ in our analysis and generate the sequence $\left\{R_n(s; z)\right\}_{n\geq 0}$ up to degree $n=15$ for both $s=3$ and $s=2$ cases.
\subsubsection{\bf{\underline{Strong Asymptotic}}}
\begin{itemize}
		\item[\rm (1)] \textbf{Asymptotic comparison:} For $s=3$, the curves $\text{Re}[\dfrac{R_n(s; z)}{z^n}]$ and limiting function are close but not identical. The gap remains small across the domain shown. However, it does not visibly shrink to zero for $n$. This suggests the existence of good pointwise approximation, but does not establish convergence. For $s=2$ case, there is a clear and persistent gap between $\text{Re}[\dfrac{R_n(s; z)}{z^n}]$ and the limit. This shows that $n\leq15$ is not large enough to make the approximation visually close.
\begin{figure}[htbp]
		\centering
		\begin{minipage}{0.80\textwidth}
			\centering
			\begin{subfigure}{0.43\textwidth}
				\centering
				\includegraphics[width=\linewidth]{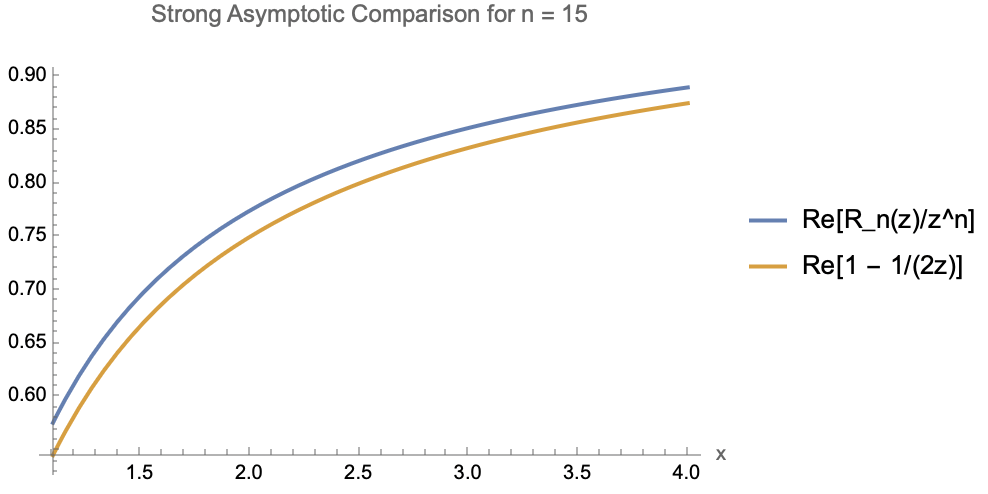}
				\label{fig:SA-comparison-s3}
			\end{subfigure}
			\hfill
			\begin{subfigure}{0.43\textwidth}
				\centering
				\includegraphics[width=\linewidth]{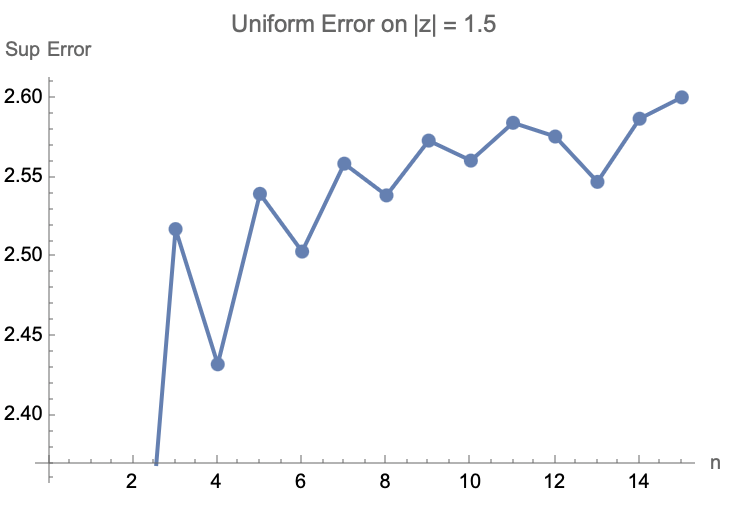}
				\label{fig:SA-uniform-error-s3}
			\end{subfigure}
			
			\vspace{-5 mm}
			
			\begin{subfigure}{0.43\textwidth}
				\centering
				\includegraphics[width=\linewidth]{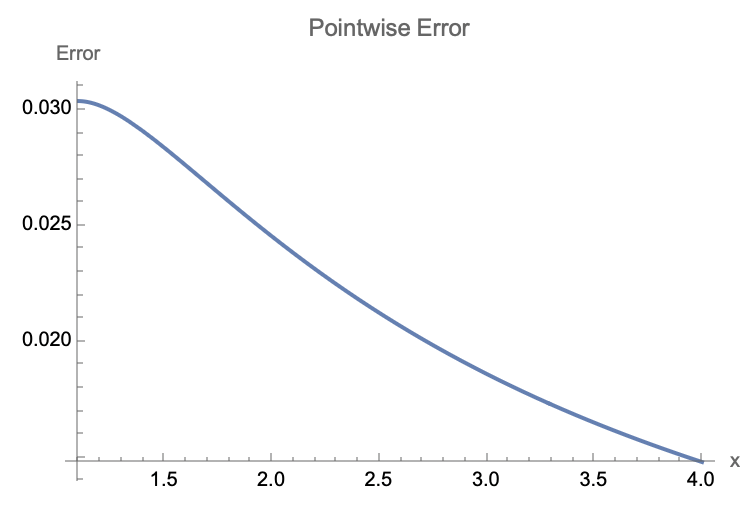}
				\label{fig:SA-pointwise-error-s3}
			\end{subfigure}
		\end{minipage}
		\vspace{-5mm}
		\caption{Strong asymptotic plots for the case $s=3$.}
		\label{Fig:SA-main-s3}
\end{figure}
\item[\rm (2)] \textbf{Supremum error:} For $s=3$, the supremum error, defined by
\begin{align*}
			E_n(s; r)= \sup_{|z|=r} \left| \dfrac{R_n(s; z)}{z^n}-f(z)\right|
\end{align*}
remains bounded and relatively stable as $n$ increases where we choose $r=1.5$. The computed error remains bounded on the tested circle $\left| z\right|=1.5$. There is no evidence of divergence or instability. However, since the error does not decrease towards zero, this does not imply uniform convergence. For $s=2$, the supremum error increases initially and then stabilizes at a larger value than in the $s=3$ case. So the approximation remains bounded and the error level is significantly higher and accuracy deteriorates compared to $s=3$.
\item[\rm (3)] \textbf{Pointwise error:} The pointwise error given by
\begin{align*}
	E_n(s;x)=\left| \dfrac{R_{n+1}(s; x)}{z^n}-f(x) \right|, \quad x>1
\end{align*}  decreases as $x$ increases for $s=3$ case. The approximation improves certain regions of the domain.  For $s=2$, we observed much weaker decay compared to $s=3$. The error starts large and decreases slowly and have persistent approximation error across the domain.
\end{itemize}
\begin{figure}[htbp]
	\centering
	\begin{minipage}{0.80\textwidth}
		\centering
		\begin{subfigure}{0.43\textwidth}
			\centering
			\includegraphics[width=\linewidth]{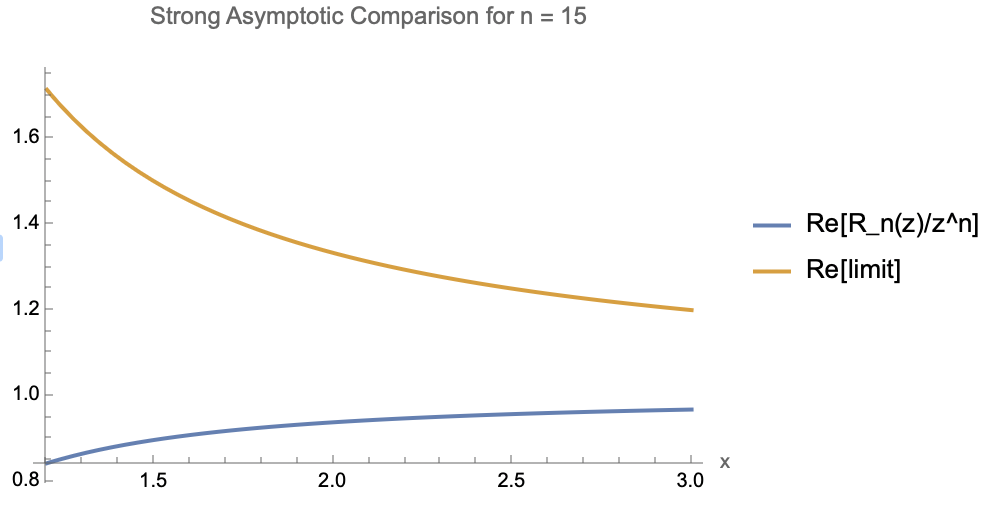}
			\label{fig:SA-comparison-s2}
		\end{subfigure}
		\hfill
		\begin{subfigure}{0.43\textwidth}
			\centering
			\includegraphics[width=\linewidth]{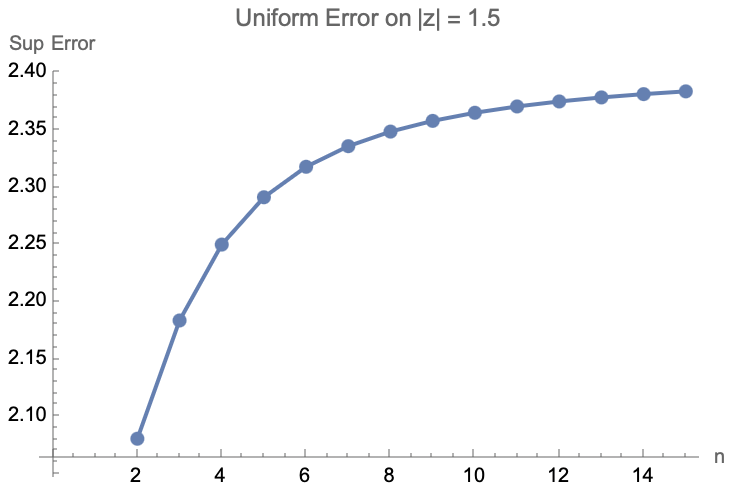}
			\label{fig:SA-uniform-error-s2}
		\end{subfigure}
		
		\vspace{-5 mm}
		
		\begin{subfigure}{0.43\textwidth}
			\centering
			\includegraphics[width=\linewidth]{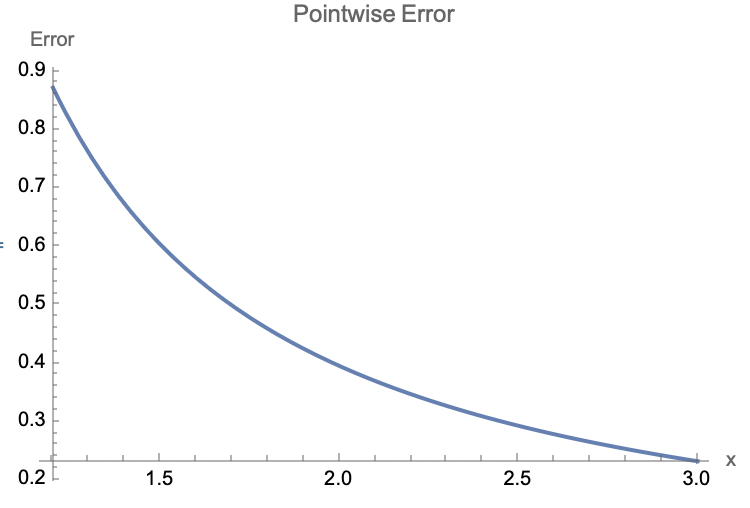}
			\label{fig:SA-pointwise-error-s2}
		\end{subfigure}
	\end{minipage}
	\vspace{-5mm}
	\caption{Strong asymptotic plots for the case $s=2$.}
	\label{Fig:SA-main-s2}
\end{figure}
\subsubsection{\bf{\underline{Ratio Asymptotic}}}
\begin{itemize}
\item[\rm (1)] \textbf{Ratio convergence:} The plots of $\text{Re} \left( \dfrac{R_{n+1}(s; z)}{R_n(s; z)}\right)$ against $x$ shows that, for $x>1$, the ratio closely follows the identity function $x$. This provides clear numerical evidence that $\dfrac{R_{n+1}(s; z)}{R_n(s; z)} \to z$ for $|z|>1$. This approximation is already visible at moderate degree $n=14$, and holds in both cases of  $s=2$ and $s=3$.
\begin{figure}[htbp]
	\centering
	\begin{minipage}{0.80\textwidth}
		\centering
		\begin{subfigure}{0.43\textwidth}
			\centering
			\includegraphics[width=\linewidth]{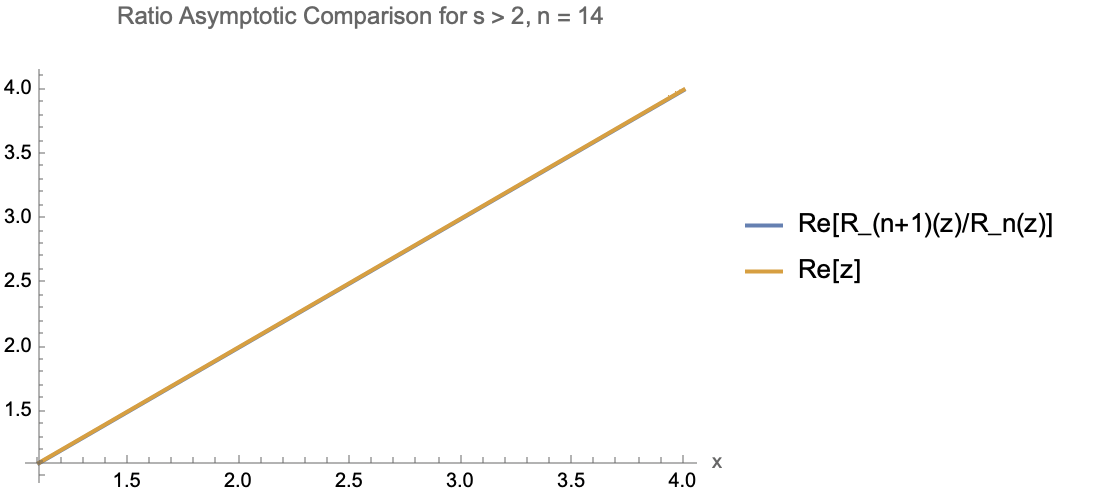}
			\label{fig:RA-comparison-s3}
		\end{subfigure}
		\hfill
		\begin{subfigure}{0.43\textwidth}
			\centering
			\includegraphics[width=\linewidth]{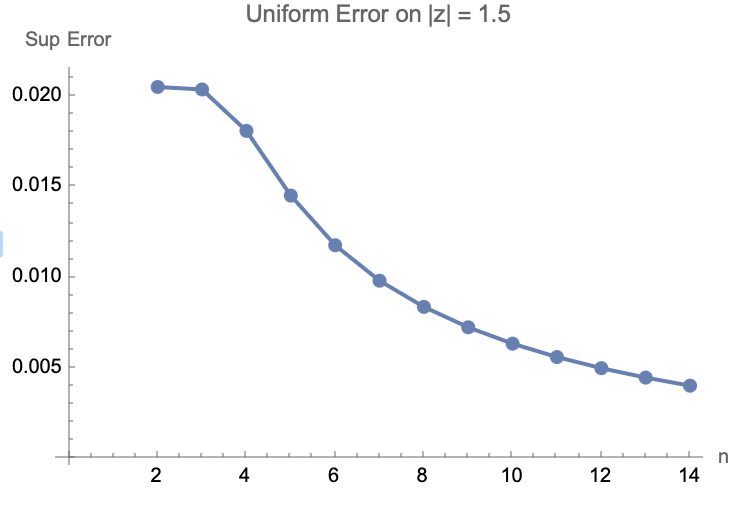}
			\label{fig:RA-uniform-error-s3}
		\end{subfigure}
		
		\vspace{-5 mm}
		
		\begin{subfigure}{0.43\textwidth}
			\centering
			\includegraphics[width=\linewidth]{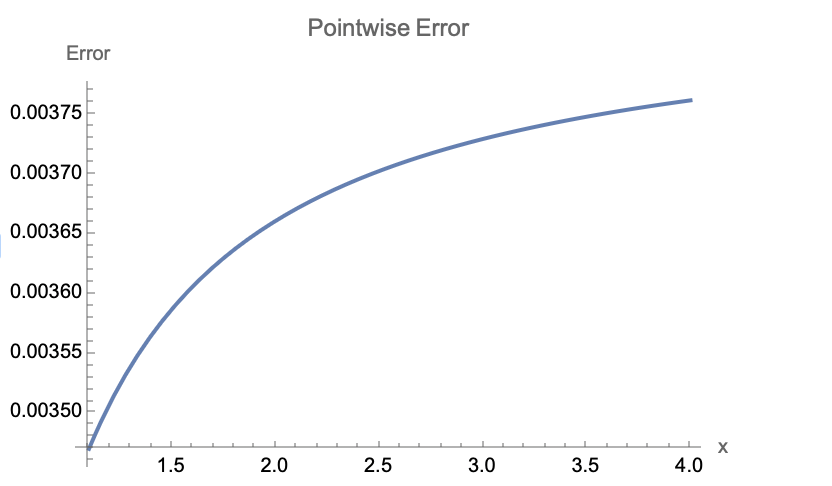}
			\label{fig:RA-pointwise-error-s3}
		\end{subfigure}
	\end{minipage}
	\vspace{-5mm}
	\caption{Ratio asymptotic plots for the case $s=3$.}
	\label{Fig:RA-main-s3}
\end{figure}
\item[\rm (2)] \textbf{Uniform error:} We computed the supremum error
\begin{align*}
E_n(s; r)= \sup_{|z|=r} \left| \dfrac{R_{n+1}(s; z)}{R_n(s; z)}-z \right|
\end{align*}
with $r=1.5$. The numerical results show $E_n(s; r)$ decreases monotonically as $n$ increases. The decay is smooth and stable  no oscillatory or unstable behavior is observed for both $s=3$ and $s=2$ cases. This provides numerical evidence for uniform convergence on the tested circle, consistent with the theoretical results for the uniform convergence on compact subsets of $|z|>1$.

\begin{figure}[htbp]
\centering
\begin{minipage}{0.80\textwidth}
\centering
\begin{subfigure}{0.43\textwidth}
	\centering
	\includegraphics[width=\linewidth]{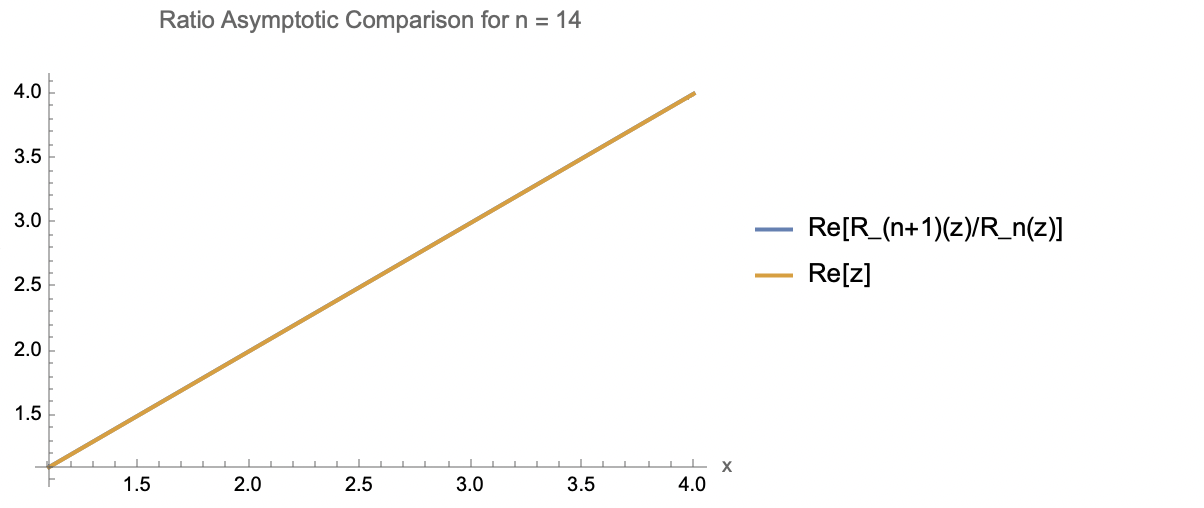}
	\label{fig:RA-comparison-s2}
\end{subfigure}
\hfill
\begin{subfigure}{0.43\textwidth}
	\centering
	\includegraphics[width=\linewidth]{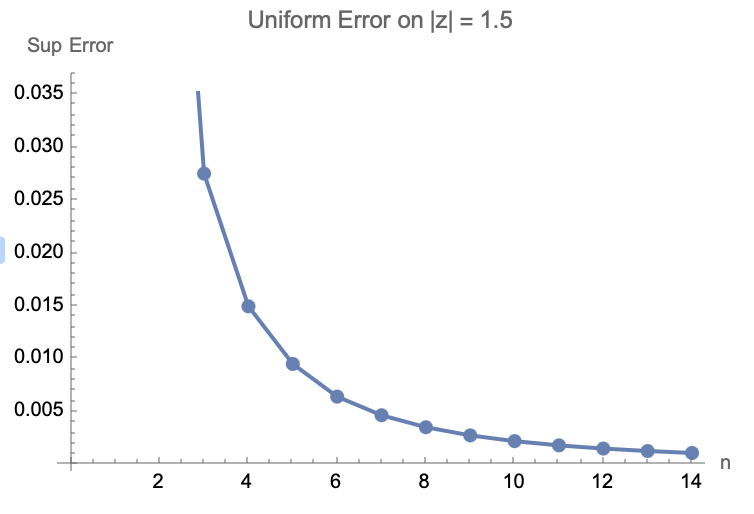}
	\label{fig:RA-uniform-error-s2}
\end{subfigure}

\vspace{-5 mm}

\begin{subfigure}{0.43\textwidth}
	\centering
	\includegraphics[width=\linewidth]{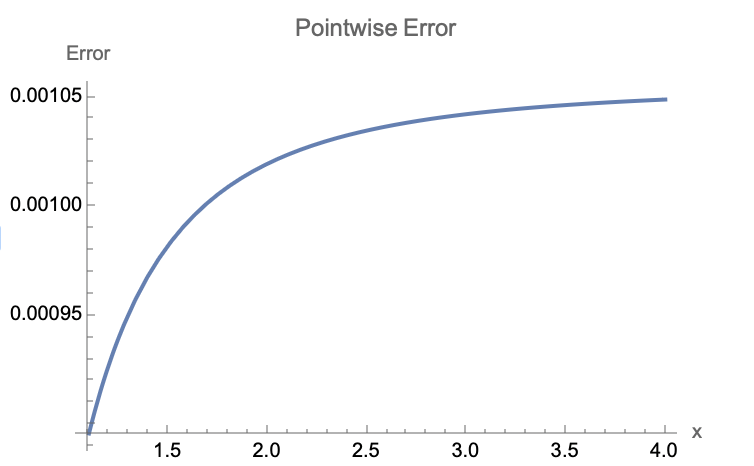}
	\label{fig:RA-pointwise-error-s2}
\end{subfigure}
\end{minipage}
\vspace{-5mm}
\caption{Ratio asymptotic plots for the case $s=2$.}
\label{fig:RA-main-s2}
\end{figure}
\item[\rm (3)] \textbf{Pointwise error:} The pointwise error $\left| \dfrac{R_{n+1}(s; x)}{R_n(s; x)}-x \right|$ remains small and varies smoothly with $x$ for both $s=3$ and $s=2$ case. No numerical instability or divergence is observed in both cases.
\end{itemize}
\subsubsection{\bf{\underline{Logarithmic Asymptotic}}}
\begin{itemize}
\item[\rm (1)] \textbf{Logarithmic asymptotic comparison:} For $x \in [1.1, 4]$, we compare $L_n(s; x):= \dfrac{1}{n}\log|R_n(s; x)|$ with $\log|x|$. In both cases $s=3$ and $s=2$, the graphs show that $L_n(s; x)$ closely follows $\log|x|$ across the entire interval. At $n=15$, the two curves are visually indistinguishable at plotting resolution. This provides strong numerical support for pointwise convergence for real argument outside the unit interval.
\begin{figure}[!h]
		\centering
		\begin{minipage}{0.80\textwidth}
			 \centering
		\begin{subfigure}{0.43\textwidth}
			\centering
			\includegraphics[width=\linewidth]{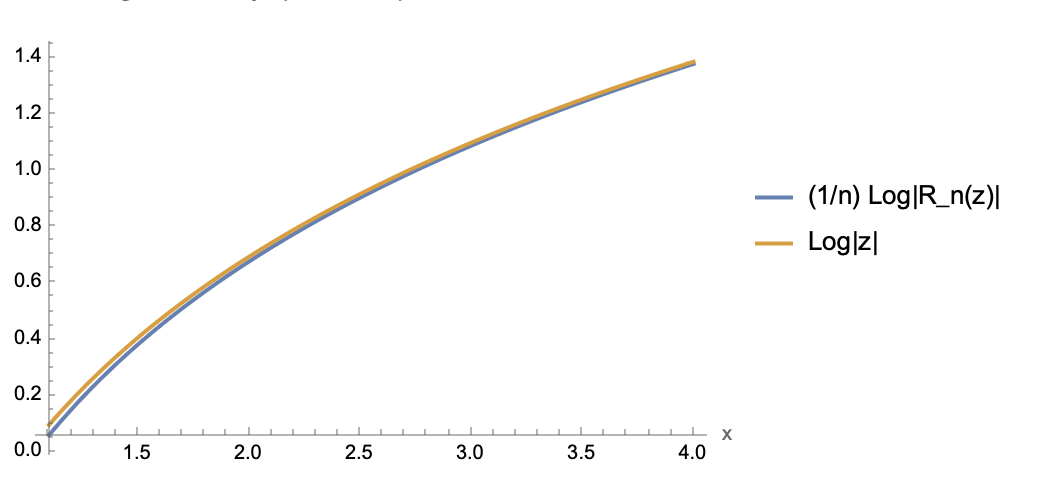}
			\label{fig:LA-comparison-s3}
		\end{subfigure}
		\hfill
		\begin{subfigure}{0.43\textwidth}
			\centering
			\includegraphics[width=\linewidth]{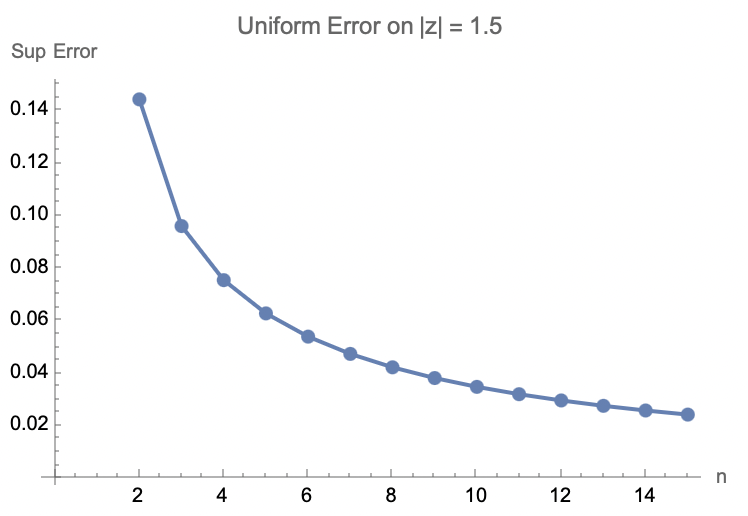}
			\label{fig:LA-uniform-error-s3}
		\end{subfigure}
		\vspace{-5mm}
		\begin{subfigure}{0.43\textwidth}
			\centering
			\includegraphics[width=\linewidth]{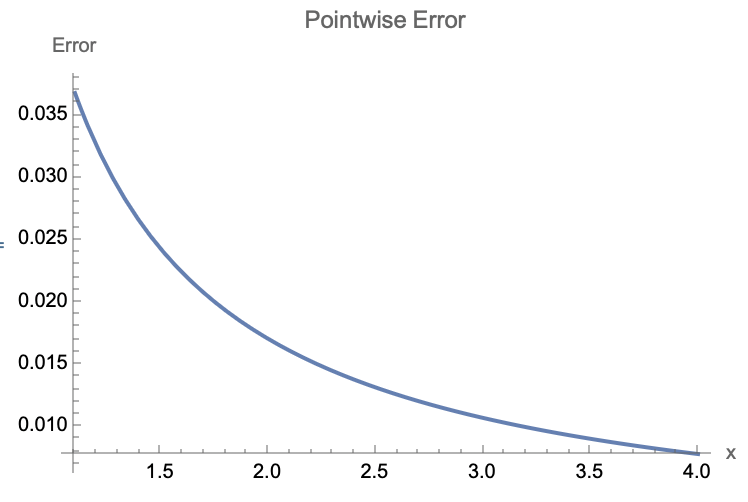}
			\label{fig:LA-pointwise-error-s3}
		\end{subfigure}
	\end{minipage}
		\vspace{-3mm}
		\caption{Logarithmic asymptotic plots for $s=3$.}
		\label{fig:LA-main-s3}
\end{figure}
\item[\rm(2)] \textbf{Uniform error on circles:} We evaluate the supremum error on the circle $|z|=1.5$, which is defined as
\begin{align*}
		E_n(s; 1.5)= \sup_{|z|=1.5}\left| \dfrac{1}{n} \log|R_n(s; z)| -\log|z|\right|.
\end{align*}
	In both cases, the numerical results show a monotone decrease of $E_n(s; 1.5)$ as $n$- increases. For $s=2$ the error decreases steadily and for $s=3$, the error also decreases, starting from a larger initial value but exhibiting constant decay. \newline These observations provide  numerical evidence of convergence in the supremum norm on the circle  $|z|=1.5$.
\begin{figure}[!hbtp]
		\centering
			\begin{minipage}{0.85\textwidth}
			\centering
		\begin{subfigure}{0.43\textwidth}
			\centering
			\includegraphics[width=\linewidth]{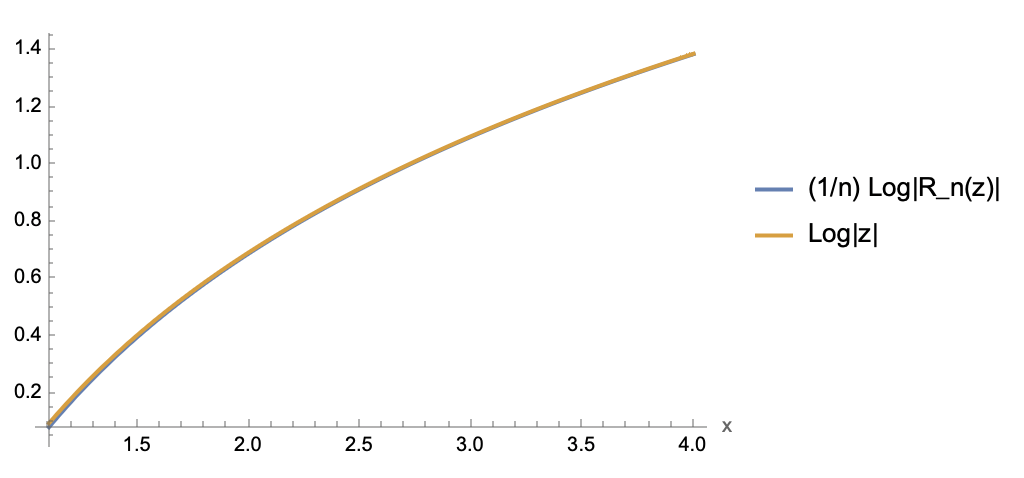}
			\label{fig:LA-comparison-s2}
		\end{subfigure}\hfill
		\begin{subfigure}{0.43\textwidth}
			\hspace{0.02\textwidth}
			\centering
			\includegraphics[width=\linewidth]{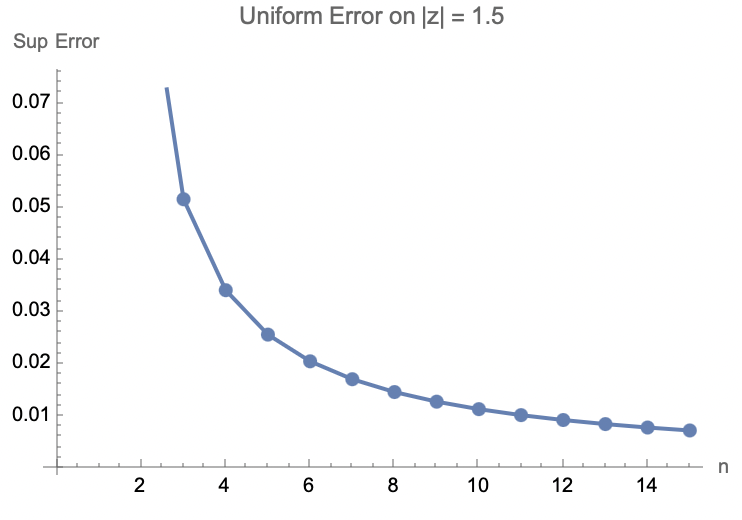}
			\label{fig:LA-uniform-error-s2}
		\end{subfigure}
		\vspace{-3mm}
		\begin{subfigure}{0.43\textwidth}
			\centering
			\includegraphics[width=\linewidth]{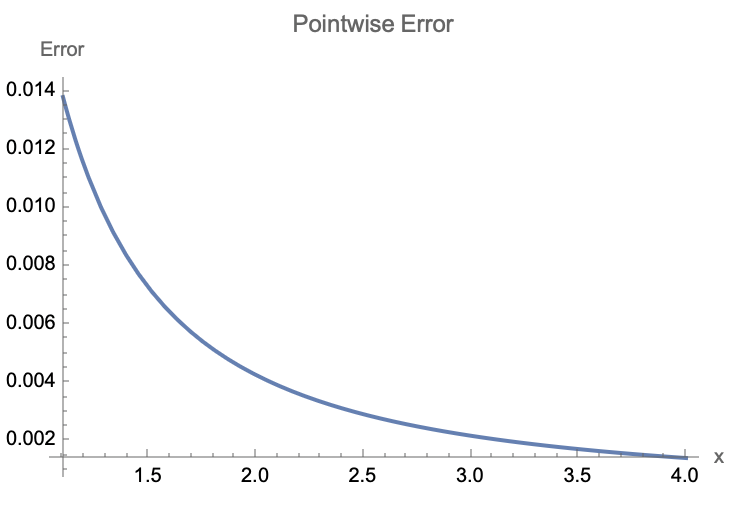}
			\label{fig:LA-pointwise-error-s2}
		\end{subfigure}
	\end{minipage}
		\vspace{-5mm}
		\caption{Logarithmic asymptotic plots for $s=2$.}
		\label{fig:LA-main-s2}
	\end{figure}
\item[\rm (3)] \textbf{Pointwise error behavior:} We consider the pointwise error
\begin{align*}
	E_n(s; x):= \left|\dfrac{1}{n} \log|R_n(s; x)| - \log|x|\right|, \quad x>1.
\end{align*}
The numerical plots shows that, for both $s=2$ and $s=3$, the error remains small across the interval $x \in [1.1, 4]$. No instability or irregular behavior is observed and the error varies smoothly as a function of $x$. The agreement between $\dfrac{1}{n} \log|R_n(s; x)|$ and $ \log|x|$ is consistently strong throughout the interval. These observation provide numerical support for pointwise convergence, without indicating an anomalous behavior in the considered range.
\end{itemize}

\noindent
{\bf{Acknowledgement.}}
This work was supported by the Project No. NBHM/RP-1/2019 of National Board for Higher Mathematics (NBHM), DAE, Government of India.

\noindent
{\bf{\underline{Data Availability}}}. Data sharing not applicable to this article as no datasets were generated or analysed during the current study.

\noindent
{\bf{\underline{Conflict of Interest}}}. There is no conflict of interest related to this manuscript.

\noindent
{\bf{\underline{Funding}}}. 
This work was supported by the Project No. NBHM/RP-1/2019 of National Board for Higher Mathematics (NBHM), DAE, Government of India.

\noindent
{\bf{\underline{Competing Interest}}}. The authors have no relevant financial or non-financial interests to disclose.

\noindent
{\bf{\underline{Authors contribution}}}. All authors contributed equally. All authors read and approved the final manuscript.

\end{document}